\documentclass[11pt]{amsart}
\usepackage{amsmath,amssymb,mathrsfs,color}

\def\X{{\mathcal{X}}}
\def\L{{\mathcal{L}}}
\def\d{{\rm{d}}}

\let\cal=\mathcal
\def\N{{\mathbb N}}

\def\R{{\mathbb R}}
\newtheorem{lemma}{Lemma}[section]
\newtheorem{theorem}[lemma]{Theorem}
\newtheorem{remark}[lemma]{Remark}
\newtheorem{prop}[lemma]{Proposition}

\newtheorem{definition}[lemma]{Definition}

\allowdisplaybreaks
\numberwithin{equation}{section}

\begin{document}

\title[recurrent solutions for SPDEs]
{Periodic, almost periodic and almost automorphic solutions for SPDEs with monotone coefficients}

\author{mengyu cheng}
\address{M. Cheng: School of Mathematical Sciences,
Dalian University of Technology, Dalian 116024, P. R. China}
\email{mengyucheng@mail.dlut.edu.cn}

\author{zhenxin liu}
\address{Z. Liu (Corresponding author): School of Mathematical Sciences,
Dalian University of Technology, Dalian 116024, P. R. China}
\email{zxliu@dlut.edu.cn}

\date{November 5, 2019}

\begin{abstract}
In this paper, we use the variational approach to investigate recurrent properties of solutions
for stochastic partial differential equations, which is in contrast to the previous semigroup framework.
Consider stochastic differential equations with monotone coefficients.
Firstly, we establish the continuous dependence on initial values and coefficients for solutions.
Secondly, we prove the existence of recurrent solutions, which
include periodic, almost periodic and almost automorphic solutions.
Then we show that these recurrent solutions are globally
asymptotically stable in square-mean sense. Finally,
for illustration of our results we give two applications, i.e. stochastic
reaction diffusion equations and stochastic porous media equations.
\end{abstract}

\maketitle

\section{Introduction}

Recurrence is an important concept in dynamical systems, which roughly means that a motion returns infinitely often
to any small neighborhood of the initial position. The recurrent phenomenon has been found in
almost all interesting systems, so it has attracted wide attention. This paper is devoted to studying the recurrence
of solutions for stochastic partial differential equations with monotone coefficients.
The types of recurrent solutions we investigate in present paper include periodic,
almost periodic and almost automorphic solutions.

The analysis of recurrent solutions to ordinary differential equations dates back to
Poincar\'e who studied periodic solutions of the three-body problem.
Later, the notion of almost periodic functions was proposed and comprehensively studied
by Bohr \cite{Bohr1924, Bohr1925, Bohr1926}. Many interesting results were acquired
in this subject; see, for example, Bochner \cite{Boch1927, Boch1962}, von Neumann \cite{vonN}
and van Kampen \cite{vanK}. After that, it was found that many differential equations,
especially equations arising from physics, possess almost periodic solutions.
So extensive investigations concerning almost periodic solutions for differential equations were conducted,
following Favard's pioneering work \cite{Fava1, Fava2};
see e.g. Amerio and Prouse \cite{AP}, Fink \cite{Fink}, Yoshizawa \cite{Yosh}, Sacker and Sell \cite{SS},
Levitan and Zhikov \cite{LZ} for a survey.
Subsequently, almost periodicity was further generalized to almost
automorphy by Bochner \cite {Boch1955}. Veech \cite{Veech1965, Veech1977}, Johnson \cite{Joh},
Shen and Yi \cite{SY}, N'Gu\'er\'ekata \cite{GM}
et al studied properties of almost automorphic functions and this kind of solutions for differential equations.

Random factors may have significant impacts on the dynamics,
so a natural question is: will recurrent phenomenon still persist when equations are perturbed by noise?
Some works have been done to prove the existence of recurrent solutions for stochastic differential
equations in both finite and infinite dimensions. For finite dimensional case, among many
other works, we mention the following which are closely related to our work.
Khasminskii \cite{Khas} investigated periodic solutions for stochastic ordinary differential equations
by Lyapunov's second method.
The existence of periodic and almost periodic solutions to
affine stochastic equations were proved by Halanay \cite{Hala}, Morozan and Tudor \cite{MT}, Arnold and Tudor \cite{AT}.
Zhao and Zheng \cite{ZZ} showed that there exist pathwise random periodic solutions to
stochastic differential equations. Liu and Wang \cite{LW} reported the existence of almost periodic solutions for
stochastic differential equations by the Favard separation method.
For infinite dimensional case, Da Prato and Tudor \cite{PT} provided the existence of periodic and almost periodic
solutions of semilinear stochastic partial differential equations.
Later, studies of periodic, almost periodic and almost automorphic solutions to
semilinear stochastic differential equations were performed
by  Bezandry and Diagana \cite{BD}, Fu and Liu \cite{FL}, Wang and Liu \cite{WL},
Chen and Lin \cite{CL}, Liu and Sun \cite{LS}, Gao \cite{GaoP}, Cheban and Liu \cite{CL1}, among others.
Note that the almost periodic/automorphic solution in \cite{BD,FL} should be in distribution sense instead of
square-mean sense, see \cite{KMR,LS} for details. It is known that the distribution of solutions
for a stochastic differential equation satisfy the corresponding
Fokker-Planck equation, so we can also study recurrent solutions through the associated Fokker-Planck equation.
See the very recent works of  Chen et al \cite{CHLY} and Ji et al \cite{JQSY} on periodic solutions to Fokker-Planck equations.

Despite considerable advances in this direction, as far as we know there is no research so far on recurrent solutions
to stochastic partial differential equations with monotone coefficients.
Note that if the equation is no longer assumed to be semilinear, it will arise that
the semigroup approach does not work any more. So a natural question is: can we still obtain recurrent
solutions for stochastic partial differential equations which are not of the semilinear form?
One of our main motivations is to partly answer this question. To this end, we adopt in this paper
the variational approach which is sometimes called monotone method, to study recurrent solutions for stochastic partial
differential equations.

Variational approach is one of basic approaches to analyze nonlinear deterministic/stochastic partial differential equations.
For deterministic partial differential equations, the approach originated from the pioneering works of
Lions \cite{Lions} and Agmon \cite{Agmon}. For stochastic partial differential equations,
the first work was done by Pardoux \cite{Pard} who proved the existence of strong solutions for
linear stochastic partial differential equations, which was based on Lions \cite{Lions}.
Subsequently, Krylov and Rozovskii \cite{KryRoz} further developed this approach to nonlinear equations
with continuous martingales as integrators. Fairly rigorous and complete description
in a slightly general form was provided by Pr\'ev\^ot and R\"ockner \cite{PR}; see also \cite{LGVM,LR}.

Now let us state the framework and our main results more precisely. Let
$(H,\langle~,~\rangle_{H})$ be a separable Hilbert space and  $H^{*}$ the dual space of $H$.
As in Zhang \cite{ZhangXC}, we assume that for each $i=1,2$, $(V_{i},\|\cdot\|_{V_{i}})$
is a reflexive Banach space such that
$V_{i}\subset H$ continuously and densely. Then we get two Gelfand triples
\[
V_{1}\subset H\subset V_{1}^{*},\quad V_{2}\subset H\subset V_{2}^{*}.
\]
Consider the following stochastic differential equation on $H$
\begin{equation}\label{maineq}
 \d X(t)=A(t,X(t))\d t +B(t,X(t))\d W(t),
\end{equation}
where $A:=A_{1}+A_{2}$, $A_{i}:\R\times V_{i}\rightarrow V_{i}^{*}$, $i=1,2$ and
$B:\R\times V\rightarrow L_{2}(U,H)$ satisfy hemicontinuous, monotone, coercive,
bounded conditions (see Section 2 for details).
Here $W(t)$, $t\in\R$ is a two-sided cylindrical Winner process on another separable Hilbert space
$(U,\langle~,~\rangle_{U})$.
Under these conditions, the existence and uniqueness of solutions
to equation \eqref{maineq} was established in \cite{ZhangXC}.
In the present paper, we first prove that the solutions of \eqref{maineq}
depend continuously on the initial value and the coefficients $A,B$, which is
useful to study qualitatively stochastic equations.
We next show that \eqref{maineq} admits a unique $L^{2}$-bounded solution when
the coefficients $A$ and $B$ satisfy some coercive and monotone conditions. Furthermore,
with the help of continuous dependence property for solutions,
we establish the recurrent properties in distribution sense for this unique $L^{2}$-bounded solution.
Indeed, it shares the same recurrent properties with the coefficients $A$ and $B$; that is,
when $A$ and $B$ are stationary (respectively, periodic, almost periodic, almost automorphic),
then so is the $L^{2}$-bounded solution in distribution sense. Then we show that this unique recurrent
(and bounded) solution is globally asymptotically stable in square-mean sense. This asymptotic stability property
is very similar to the ergodicity of homogeneous Markov processes; note that the coefficients $A$ and $B$ depend on
$t$, so equation \eqref{maineq} generates an inhomogeneous Markov process. Finally,
to illustrate the theoretical results obtained above, we discuss two examples, i.e. stochastic reaction
diffusion equations and stochastic porous media equations.

The remainder of this paper is organized as follows.
Section 2 gives some definitions and properties of recurrent functions
as well as a rough introduction to variational approach.
In Section 3, we obtain continuous dependence on initial values and coefficients for solutions
of \eqref{maineq}, and prove that \eqref{maineq} admits a unique $L^{2}$-bounded solution under suitable conditions.
In Section 4, we show that the $L^{2}$-bounded solution has the same recurrent properties as the coefficients.
In Section 5, we consider the additive noise case. In this situation, the strictly monotone condition
can be weakened to the strong monotone condition. Section 6 discusses global asymptotic
stability of the $L^{2}$-bounded solution. In the last section, we give two applications.

\section{Preliminaries}

Before turning to our results, we first give some preliminaries.
Let $(\cal{X},d)$ be a complete metric space. We write $C(\R,\cal{X})$ to mean the space of
all continuous functions $\varphi:\R\rightarrow\cal{X}$.

\subsection{Recurrent functions.}

Let us now recall some types of recurrent functions to be studied in this paper.

\begin{definition}\rm\label{tperiod}
We say $\varphi\in C(\mathbb R,\cal{X})$ is {\em T-periodic}, if there exists some nonzero constant
$T\in\R$ such that $\varphi(t+T)=\varphi(t)$ for all $t\in\R$. In particular,
$\varphi\in C(\mathbb R,\cal{X})$ is called {\em stationary}  provided
$\varphi(t)=\varphi(0)$ for all $t\in\R$.
\end{definition}

\begin{definition} \rm\label{Bohrap}
We say $\varphi \in C(\mathbb R,\cal{X})$ is {\em Bohr almost periodic} if the set $\mathcal T(\varphi,\varepsilon)$
of $\varepsilon$-almost periods of $\varphi$ is relatively dense for each $\varepsilon >0$, i.e. there exists
a constant $l=l(\varepsilon)>0$ such that
$\mathcal T(\varphi,\varepsilon)\cap [a,a+l]\not=\emptyset$ for all $a\in\mathbb R$, where
$$\mathcal T(\varphi,\varepsilon):=\{\tau\in\R:\sup\limits_{t\in\R}d(\varphi(t+\tau),\varphi(t))<\varepsilon\}.$$
\end{definition}

Let $\varphi$ be a mapping from $\mathbb{R}$ to $\cal{X}$. We employ $\gamma$ to denote a sequence
$\{\gamma_{n}\}:=\{\gamma_{n}\}_{n=1}^{\infty}$ in $\R$. Denote
$(T_{\gamma}\varphi)(\cdot):=\lim\limits_{n\rightarrow\infty}\varphi(\cdot+\gamma_{n})$, provided the limit exists.
The mode of convergence will be pointed out when this symbol is used.
Recall the following characterization of almost periodicity that is due to
Bochner \cite{Boch1927}.

\begin{definition}\rm
We say $\varphi\in C(\mathbb{R},\cal{X})$ is {\em Bochner almost periodic},
if for any sequence $\gamma'=\{\gamma'_{n}\}\subset\mathbb{R}$ there exists a subsequence
$\gamma=\{\gamma_{n}\}\subset\gamma'$ such that $T_{\gamma}\varphi$ exists uniformly on $\R$.
\end{definition}

\begin{theorem}[Bochner]\label{BochBohr}
Assume that $\varphi:\R\rightarrow \X$ is continuous. Then the following statements are equivalent.
\begin{enumerate}
  \item $\varphi$ is Bohr almost periodic.
  \item $\varphi$ is Bochner almost periodic.
  \item For any two sequences $\gamma'=\{\gamma'_{n}\}\subset\mathbb R$ and
         $\beta'=\{\beta'_{n}\}\subset\mathbb R$ there exist two subsequences $\gamma=\{\gamma_{n}\}\subset\gamma'$
         and $\beta=\{\beta_{n}\}\subset\beta'$ with the same indexes such that
         $$T_{\gamma+\beta}\varphi=T_{\gamma}T_{\beta}\varphi$$
        uniformly on $\mathbb R$.
  \item  For any two sequences $\gamma'=\{\gamma'_{n}\}\subset\mathbb R$ and
         $\beta'=\{\beta'_{n}\}\subset\mathbb R$ there exist two subsequences $\gamma=\{\gamma_{n}\}\subset\gamma'$
         and $\beta=\{\beta_{n}\}\subset\beta'$ with the same indexes such that
         $$T_{\gamma+\beta}\varphi=T_{\gamma}T_{\beta}\varphi$$
         pointwise.
\end{enumerate}
\end{theorem}

\begin{remark}\rm
It follows from  the above theorem that Bohr's almost periodicity is equivalent to Bochner's one.
Therefore, we just call them almost periodicity below.
\end{remark}

\begin{definition}\rm\label{Bochaa}
We say $\varphi\in C(\mathbb{R},\X)$ is {\em almost automorphic},
if for any sequence $\gamma'=\{\gamma'_{n}\}\subset\mathbb{R}$ there exist a subsequence
$\gamma=\{\gamma_{n}\}\subset\gamma'$ and some function $\psi:\mathbb{R}\rightarrow \X$
such that
$$\lim_{n\rightarrow\infty}\varphi(t+\gamma_{n})=\psi(t)\quad {\rm{and}}\quad
\lim_{n\rightarrow\infty}\psi(t-\gamma_{n})=\varphi(t)$$
pointwise for $t\in\mathbb{R}$.
\end{definition}

In order to study recurrent solutions of differential equations, we need to recall the definition of uniformly almost
periodic/automorphic functions. Let $(\X_{i},d_{i})$, $i=1,2$ be complete metric spaces.
We write $C(\R\times\X_{1},\X_{2})$ to mean the set of all continuous functions
$\varphi:\R\times\X_{1}\rightarrow\X_{2}$.

\begin{definition}\rm
We say $\varphi \in C(\mathbb R\times\X_{1},\cal{X}_{2})$ is {\em T-periodic in t},
if there exists some nonzero constant $T\in\R$ such that $\varphi(t+T,x)=\varphi(t,x)$
for all $t\in\R$, $x\in\X_{1}$.
\end{definition}

\begin{definition} [See Yoshizawa \cite{Yosh}]\rm
We say $\varphi \in C(\mathbb R\times\X_{1},\cal{X}_{2})$ is
{\em almost periodic in t uniformly for $x\in\X_{1}$} if for any $\varepsilon >0$
and any compact set $Q\subset\X_{1}$, the set $\mathcal T(\varphi,\varepsilon,Q)$ is relatively dense,
i.e. there exists a constant $l=l(\varepsilon,Q)>0$ such that
$\mathcal T(\varphi,\varepsilon,Q)\cap [a,a+l]\not=\emptyset$ for all $a\in\mathbb R$, where
$$\mathcal T(\varphi,\varepsilon,Q):=
\{\tau\in\R: \sup\limits_{(t,x)\in\R\times Q}d_{2}(\varphi(t+\tau,x),\varphi(t,x))<\varepsilon\}.$$
\end{definition}

Similar to Theorem \ref{BochBohr}, we recall the following results.

\begin{theorem}[See Yoshizawa \cite{Yosh}]
Suppose that $\varphi:\mathbb R\times\X_{1}\rightarrow\cal{X}_{2}$ is continuous. Then the
following statements are equivalent.
\begin{enumerate}
  \item $\varphi$ is almost periodic in $t$ uniformly for $x\in\X_{1}$.
  \item For any sequence $\gamma'=\{\gamma_{n}'\}\subset\R$ there exists a  subsequence
        $\gamma=\{\gamma_{n}\}\subset\gamma'$ such that
        $$(T_{\gamma}\varphi)(t,x):=\lim_{n\rightarrow\infty}\varphi(t+\gamma_{n},x)$$
        exists uniformly with respect to $t\in\R$ and $x\in Q$, where $Q$ is an arbitrary compact
        subset of $\X_{1}$.
  \item For any two sequences $\gamma'=\{\gamma_{n}'\}\subset\R$ and $\beta'=\{\beta_{n}'\}\subset\R$
        there exist two subsequences $\gamma=\{\gamma_{n}\}\subset\gamma'$ and
        $\beta=\{\beta_{n}\}\subset\beta'$ with the same indexes such that
        $$T_{\gamma+\beta}\varphi=T_{\gamma}T_{\beta}\varphi$$
        uniformly on $\R\times Q$, where $Q$ is an arbitrary compact subset of $\X_{1}$.
\end{enumerate}
\end{theorem}

\begin{definition}[See  Shen and Yi \cite{SY}]\rm
We say $\varphi \in C(\mathbb R\times\X_{1},\cal{X}_{2})$ is
{\em almost automorphic in t uniformly for $x\in\X_{1}$}, if for any sequence
$\gamma'=\{\gamma_{n}'\}\subset\R$
there exist a subsequence $\gamma=\{\gamma_{n}\}\subset\gamma'$ and some  function
$\psi:\R\times\X_{1}\rightarrow\X_{2}$ such that
$$\lim_{n\rightarrow\infty}\varphi(t+\gamma_{n},x)=\psi(t,x) \quad {\rm{and}}\quad
\lim_{n\rightarrow\infty}\psi(t-\gamma_{n},x)=\varphi(t,x)$$
uniformly on $[a,b]\times Q$, where $[a,b]$ is an arbitrary finite interval and $Q$ an arbitrary
compact subset of $\X_{1}$.
\end{definition}

\begin{remark}\rm
When we consider stochastic partial differential equations, $\X_{1}$ is a Banach space $V$.
For the sake of simplicity, we call a function $\varphi$ ``uniformly almost periodic
(uniformly almost automorphic)'', which means that $\varphi$ is almost periodic (almost automorphic) in
$t$ uniformly  for $x\in V$.
\end{remark}

\subsection{Recurrence in distribution.}

Suppose further that $(\X,d)$ is a Polish space, i.e. a separable complete metric space. We write $Pr(\X)$ to mean the set of all
Borel probability measures on $\X$. Denote by $C_{b}(\X)$ the space of all continuous functions
$\varphi:\X\rightarrow \R$ for which the norm
$\|\varphi\|_{\infty}:=\sup\limits_{x\in \X}|\varphi(x)|$
is finite. Let $\{\mu_{n}\}:=\{\mu_{n}\}_{n=1}^{\infty}\subset Pr(\X)$ and $\mu\in Pr(\X)$. We say
$\mu_{n}$ {\em converges weakly} to $\mu$ in $Pr(\X)$, provided $\int \varphi\d\mu_{n}$
converges to $\int \varphi\d\mu$ for all $\varphi\in C_{b}(\X)$. Let $\varphi\in C_{b}(\X)$
be Lipschitz continuous, we define
\begin{equation*}
  \|\varphi\|_{BL}:=Lip(\varphi)+\|\varphi\|_{\infty},
\end{equation*}
where $Lip(\varphi)=\sup\limits_{x\neq y}\frac{|\varphi(x)-\varphi(y)|}{d(x,y)}$. We endow $Pr(\X)$
with $d_{BL}$ metric, where
\[
d_{BL}(\mu,\nu):=\sup\left\{\left|\int \varphi\d\mu-\int \varphi\d\nu\right|:\|\varphi\|_{BL}\leq1\right\}
\]
for all $\mu,~\nu\in Pr(\X)$.
It is well known that $d_{BL}$ generates the weak topology on $Pr(\X)$, i.e.
$\mu_{n}\rightarrow\mu$ weakly in $Pr(\X)$ if and only if
$d_{BL}(\mu_{n},\mu)\rightarrow0$ as $n\rightarrow\infty$. See Chapter 11 in \cite{Dudley} for this metric
$d_{BL}$ (denoted by $\beta$ there) and its related properties.

We assume in the following exposition that $(\Omega,\mathcal{F},P)$ is a complete probability space and that
$(H,\langle~,~\rangle_{H})$ is a separable Hilbert space. The space $L^{2}(\Omega,P;H)$
consists of all $H$-valued random variables $\zeta$ such that
$E\|\zeta\|_{H}^{2}=\int_{\Omega}\|\zeta\|_{H}^{2} \d P<\infty$. We say an $H$-valued stochastic process
$X=\{X(t): t\in\R\}$ is {\em $L^{2}$-bounded} provided $\sup\limits_{t\in\R}E\|X(t)\|_{H}^{2}<\infty$.
Throughout the paper, we denote by $\L(\zeta)\in Pr(H)$ the law or distribution of
$H$-valued random variable $\zeta$. A sequence of $H$-valued continuous stochastic processes $\{X_{n}\}$ is said to
{\em converge in distribution to $X$ (on $C(\R,H)$)} provided $\L(X_{n})$ weakly converges to $\L(X)$ in $Pr(C(\R,H))$,
where $\L(X)$ is the law or distribution of $X$ on $C(\R,H)$; if $d_{BL} (\L (X_n(t)), \L(X(t)))\to 0$
as $n\to\infty$ for each $t\in\R$, we simply say that $X_n$ {\em converges in distribution to $X$ on $H$}.

Note that $(Pr(C(\R,H)),d_{BL})$ and $(Pr(H),d_{BL})$ are Polish spaces (see, e.g. Theorems 6.2 and 6.5 in
Chapter \uppercase\expandafter{\romannumeral2} of \cite{Parth}). So, similar to Definitions \ref{tperiod}, \ref{Bohrap}
and \ref{Bochaa}, we can define recurrence in distribution as follows.

\begin{definition}\rm\label{apaalaw}
We say an $H$-valued continuous stochastic process $X$ is {\em T-periodic
(respectively, almost periodic, almost automorphic) in distribution},
if the mapping $t\mapsto \L(X(t+\cdot))$ is $T$-periodic (respectively, almost periodic, almost automorphic) in $Pr(C(\R,H))$.
In particular, $X$ is called {\em stationary} provided $X$ is $T$-periodic in distribution for any $T\in\R$.
\end{definition}

\begin{remark}\rm
Let $X$ be an $H$-valued continuous stochastic process. Note that
$\mu(t):=\L(X(t))$, $t\in\R$ is $T$-periodic (respectively, almost periodic, almost automorphic)
in $Pr(H)$, provided $\L(X)$ is $T$-periodic (respectively, almost periodic, almost automorphic)
in $Pr(C(\R,H))$. But the converse is not true in general.
\end{remark}

\subsection{Variational approach}

Recall that $H$ is a separable Hilbert space with norm $\|\cdot\|_{H}$ and inner product
$\langle~,~\rangle_{H}$, and that $H^{*}$ is the dual space of $H$. Let $(V,\|\cdot\|_{V})$ be a
reflexive Banach space such that $V\subset H$ continuously and densely. So we have $H^{*}\subset V^{*}$
continuously and densely. Identifying $H$ with its dual $H^{*}$ via
the Riesz isomorphism, then we have
$$V\subset H\subset V^{*}$$
continuously and densely. We write $_{V^{*}}\langle~,~\rangle_{V}$ to denote the pairing between $V^{*}$
and $V$. It follows that
$$ _{V^{*}} \langle h,v\rangle_{V}=\langle h,v\rangle_{H}$$
for all $h\in H$, $v\in V$. $(V,H,V^{*})$ is called {\em Gelfand triple}.
Since $H\subset V^{*}$ continuously and densely,
we deduce that $V^{*}$ is separable, hence so is $V$.
See \cite{PR} for details.

Assume that $(V_{1},\|\cdot\|_{V_{1}})$ and $(V_{2},\|\cdot\|_{V_{2}})$ are reflexive Banach spaces
and embedded in $H$ continuously and densely. Then we get two triples:
\[
V_{1}\subset H\backsimeq H^{*}\subset V_{1}^{*}\quad {\rm{and}}\quad
V_{2}\subset H\backsimeq H^{*}\subset V_{2}^{*}.
\]
We define the norm $\|v\|_{V}:=\|v\|_{V_{1}}+\|v\|_{V_{2}}$ on the space $V:=V_{1}\cap V_{2}$.
Note that $(V,\|\cdot\|_{V})$ is also a Banach space. Since $V_{1}^{*}$ and $V_{2}^{*}$
can be thought as subspaces of $V^{*}$, we get a Banach space
$W:=V_{1}^{*}+V_{2}^{*}\subset V^{*}$ with norm
\[
\|f\|_{W}:=\inf\left\{\|f_{1}\|_{V_{1}^{*}}+\|f_{2}\|_{V_{2}^{*}}:
f=f_{1}+f_{2}, ~f_{i}\in V_{i}^{*}, ~i=1,2\right\}.
\]
Similarly, we write $_{V_{i}^{*}}\langle~,~\rangle_{V_{i}}$ to denote the pairing between $V_{i}^{*}$
and $V_{i}$, $i=1,2$. Then, for all $v\in V$ and $f=f_{1}+f_{2}\in W\subset V^{*}$ we have
\[
_{V^{*}}\langle f,v \rangle_{V}= ~_{V_{1}^{*}}\langle f_{1},v \rangle_{V_{1}}
+~_{V_{2}^{*}}\langle f_{2},v \rangle_{V_{2}}.
\]
Note carefully that if $f\in H$ and $v\in V$, then we obtain
\[
_{V^{*}}\langle f,v \rangle_{V}= ~_{V_{1}^{*}}\langle f,v \rangle_{V_{1}}
=~_{V_{2}^{*}}\langle f,v \rangle_{V_{2}}=\langle f,v \rangle_{H}.
\]

Let $W(t)$, $t\in\R$ be a two-sided cylindrical $Q$-Wiener process with $Q=I$ on
a separable Hilbert space $(U,\langle~,~\rangle_{U})$ with respect to a complete
filtered probability space $(\Omega,\mathcal{F},\mathcal{F}_{t},P)$. Denote by
$L_{2}(U,H)$ the space of all Hilbert-Schmidt operators from $U$ into $H$. Consider the following
stochastic partial differential equation on $H$
\begin{equation}\label{eqSPDE1}
\ \d X(t)=A(t,X(t))\d t+B(t,X(t))\d W(t),
\end{equation}
where $A=A_{1}+A_{2}$, $A_{i}:\R\times V_{i}\rightarrow V_{i}^{*}$, $i=1,2$ and
$B:\R\times V\rightarrow L_{2}(U,H)$.

Let us introduce the following conditions.
\begin{enumerate}

\item[(H1)](Hemicontinuity) For all $u$, $v$, $w \in V$ and $t\in \mathbb{R}$ the map
$$\mathbb{R}\ni \theta \mapsto~_{V^{*}}\langle A(t,u+\theta v),w\rangle_{V}$$
is continuous;

\item[(H2)](Monotonicity) There exists a constant $c\in\R$ such that for all $u$, $v\in V$, $t\in \mathbb{R}$
\begin{align*}
    2_{V^{*}}\langle A(t,u)-A(t,v),u-v\rangle_{V}+\|B(t,u)-B(t,v)\|^{2}_{L_{2}(U,H)}   \leq c\|u-v\|^{2}_{H};
\end{align*}

\item[(H3)](Coercivity) There  exist constants $\alpha_{1}$, $\alpha_{2} \in(1,\infty)$,
           $c_{1}\in \mathbb{R}$, $c_{2}$, $c_{2}'\in(0,\infty)$ and $M_{0}\in(0,\infty)$
          such that for all $v\in V$, $t\in\mathbb{R}$
 \begin{equation*}
   2_{V^{*}}\langle A(t,v),v\rangle_{V}+\|B(t,v)\|^{2}_{L_{2}(U,H)}\leq c_{1}\|v\|^{2}_{H}
   -c_{2}\|v\|^{\alpha_{1}}_{V_{1}}-c_{2}'\|v\|^{\alpha_{2}}_{V_{2}}+M_{0} ;
 \end{equation*}

\item[(H4)](Boundedness) There exist constants $c_{3}$, $c_{3}'\in(0,\infty)$ such that for
    all $v\in V$, $t\in\mathbb{R}$
    $$\|A_{1}(t,v)\|_{V_{1}^{*}}\leq c_{3}\|v\|^{\alpha_{1}-1}_{V_{1}}+M_{0},$$
    $$\|A_{2}(t,v)\|_{V_{2}^{*}}\leq c_{3}'\|v\|^{\alpha_{2}-1}_{V_{2}}+M_{0},$$
    where $\alpha_{i}$ is as in (H3).
\end{enumerate}

\begin{definition} \label{vsolution}\rm
We say  continuous $H$-valued $(\mathcal{F}_{t})$-adapted process $X(t)$, $t\in[0,T]$ is a
{\em solution} to equation \eqref{eqSPDE1},
if for its $\d t\otimes P$-equivalence class $\widehat{X}$ we have
$\widehat{X}\in \cap_{i=1,2}L^{\alpha_{i}}([0,T]\times \Omega,\d t\otimes P;V_{i})
\cap L^{2}([0,T]\times \Omega, \d t\otimes P;H)$
with $\alpha_{i}$ as in (H3) and $P$-a.s.
\begin{equation}
\ X(t)=X(s)+\int^{t}_{s}A(\sigma,\overline{X}(\sigma))\d\sigma
+\int^{t}_{s}B(\sigma,\overline{X}(\sigma))\d W(\sigma), \quad 0\leq s\leq t\leq T,
\end{equation}
where $\overline{X}$ is any $V$-valued progressively measurable $\d t\otimes P$-version of $\widehat{X}$.
\end{definition}

\begin{remark}\label{estB}\rm
\begin{enumerate}
  \item Note that solutions in Definition \ref{vsolution} are usually called variational
        solutions in the literature.
  \item By (H3) and (H4),  for all $t\in \R$ and $v\in V$ we have
  \begin{align*}
    \|B(t,v)\|_{L_{2}(U,H)}^{2}  & \leq c_{1}\|v\|_{H}^{2}+\left(2c_{3}-c_{2}\right)\|v\|_{V_{1}}^{\alpha_{1}}
    +\left(2c'_{3}-c'_{2}\right)\|v\|_{V_{2}}^{\alpha_{2}} \\
     & \quad+2M_{0}\left(\|v\|_{V_{1}}+\|v\|_{V_{2}}\right)+M_{0}.
  \end{align*}
  \item Suppose that (H1)--(H4) hold, then for any $X_{0}\in L^{2}(\Omega,\mathcal{F}_{0},P;H)$
        there exists a unique solution to equation \eqref{eqSPDE1}
        in the sense of Definition \ref{vsolution} (see \cite{ZhangXC} for more general results).
  \item For all $0\leq s \leq t \leq T$ we have the following It\^o's formula
        (see, e.g. \cite[Theorem 4.2.5]{PR}).
  \begin{align*}
    \|X(t)\|_{H}^{2} & =\|X(s)\|_{H}^{2}+\int^{t}_{s}\left(2_{V^{*}}\langle A(\sigma,\overline{X}(\sigma)),
    \overline{X}(\sigma)\rangle_{V}+\|B(\sigma,\overline{X}(\sigma))\|_{L_{2}(U,H)}^{2}\right) \d\sigma\\
     & \quad +2\int^{t}_{s}\langle X(\sigma),B(\sigma,\overline{X}(\sigma))\d W(\sigma)\rangle_{H} .
  \end{align*}
\end{enumerate}
\end{remark}

\section{Continuous dependence and bounded solutions}

The following result, which shows continuous dependence on initial values and
coefficients for solutions to equation \eqref{eqSPDE1}, is interesting on its own rights,
so we state it as a theorem. It turns out to be necessary to consider the following condition.

(HL) There exists a constant $L_{B}>0$ such that for all $u$, $v\in V$, $t\in \mathbb{R}$
$$\|B(t,u)-B(t,v)\|_{L_{2}(U,H)}\leq L_{B}\|u-v\|_{H}.$$

\begin{theorem}\label{conlemma}
Suppose that $A_{n}$, $A$, $B_{n}$, $B$ satisfy {\rm{(H1)--(H4)}} and {\rm{(HL)}} with the same constants
$c$, $c_{1}$, $c_{2}$, $c_{3}$, $c_{2}'$, $c_{3}'$, $M_{0}$, $\alpha_{i}$, $i=1,2$ and $L_{B}$.
Let $X_{n}$ be a solution of the Cauchy problem
\begin{equation}\label{system1}
  \left\{
   \begin{aligned}
   &\ \d X(t)=A_{n}(t,X(t))\d t+B_{n}(t,X(t))\d W(t)\\
  &\ X(s)=\zeta_{n}^{s}
   \end{aligned}
   \right.
  \end{equation}
and $X$ be a solution to the Cauchy problem
\begin{equation}\label{system2}
  \left\{
   \begin{aligned}
   &\ \d X(t)=A(t,X(t))\d t+B(t,X(t))\d W(t)\\
  &\ X(s)=\zeta^{s}.
   \end{aligned}
   \right.
  \end{equation}
Assume further that
\begin{enumerate}
  \item[(1)] $\lim\limits_{n\rightarrow\infty}A_{i,n}(t,x)=A_{i}(t,x)~in~V_{i}^{*}~for~all~t\in\mathbb{R},~x\in V,~i=1,2$;
  \item[(2)] $\lim\limits_{n\rightarrow\infty}B_{n}(t,x)=B(t,x)~in~L_{2}(U,H)~for~all~t\in\mathbb{R},~x\in V$.
\end{enumerate}
Then we have the following conclusions:
\begin{enumerate}
  \item If $\lim\limits_{n\rightarrow\infty}E\|\zeta_{n}^{s}-\zeta^{s}\|^{2}_{H}=0$,
  then $\lim\limits_{n\rightarrow\infty}E\sup\limits_{s\leq\tau\leq t}\|X_{n}(\tau)-X(\tau)\|^{2}_{H}=0$ for any $t>s$;
  \item If $\lim\limits_{n\rightarrow\infty}\zeta_{n}^{s}=\zeta^{s}$ in probability,
  then $\lim\limits_{n\rightarrow\infty}\sup\limits_{\tau\in[s,t]}\|X_{n}(\tau)-X(\tau)\|_{H}=0$ in probability;
  \item If $\lim\limits_{n\rightarrow\infty}d_{BL}(\mathcal{L}(\zeta_{n}^{s}),\mathcal{L}(\zeta^{s}))=0$ in $Pr(H)$,
  then
  $$\lim\limits_{n\rightarrow\infty}d_{BL}(\mathcal{L}(X_{n}),\mathcal{L}(X))=0 \quad {\rm{in}} ~Pr(C([s,\infty),H)).$$
\end{enumerate}
\end{theorem}

\begin{proof}
(i) Employing It\^o's formula, we obtain
\begin{align*}
   & E\sup_{s\leq\tau\leq t}\|X_{n}(\tau)-X(\tau)\|^{2}_{H} \\
   & \leq E\|\zeta_{n}^{s}-\zeta^{s}\|^{2}_{H}
     +E\sup_{s\leq\tau\leq t}\int^{\tau}_{s}\Big(2_{V^{*}}\langle A_{n}(\sigma,\overline{X}_{n}(\sigma))
     -A_{n}(\sigma,\overline{X}(\sigma)),\overline{X}_{n}(\sigma)-\overline{X}(\sigma)\rangle_{V} \\
   & \qquad +2\|B_{n}(\sigma,\overline{X}_{n}(\sigma))-B_{n}(\sigma,\overline{X}(\sigma))\|^{2}_{L_{2}(U,H)}\Big)\d\sigma \\
   & \quad +E\sup_{s\leq\tau\leq t}\int^{\tau}_{s}\Big(2_{V^{*}}\langle A_{n}(\sigma,\overline{X}(\sigma))
     -A(\sigma,\overline{X}(\sigma)),\overline{X}_{n}(\sigma)-\overline{X}(\sigma)\rangle_{V}\\
   & \qquad +2\|B_{n}(\sigma,\overline{X}(\sigma))-B(\sigma,\overline{X}(\sigma))\|^{2}_{L_{2}(U,H)}\Big)\d\sigma \\
   & \quad +E\sup_{s\leq\tau\leq t}2\int^{\tau}_{s}\langle X_{n}(\sigma)-X(\sigma),
     \left[B_{n}(\sigma,\overline{X}_{n}(\sigma))-B_{n}(\sigma,\overline{X}(\sigma))\right]\d W(\sigma)\rangle_{H} \\
   & \quad +E\sup_{s\leq\tau\leq t}2\int^{\tau}_{s}\langle X_{n}(\sigma)-X(\sigma),
     \left[B_{n}(\sigma,\overline{X}(\sigma))-B(\sigma,\overline{X}(\sigma))\right]\d W(\sigma)\rangle_{H}\\
   & =:E\|\zeta_{n}^{s}-\zeta^{s}\|^{2}_{H}+\mathbb{I}_{1}+\mathbb{I}_{2}+\mathbb{I}_{3}+\mathbb{I}_{4}.
\end{align*}
For the first term $\mathbb{I}_{1}$, by (H2) and (HL) we have
\begin{align}\label{I1}
  \mathbb{I}_{1} & := E\sup_{s\leq\tau\leq t}\int^{\tau}_{s}\Big(2_{V^{*}}\langle A_{n}(\sigma,\overline{X}_{n}(\sigma))
      -A_{n}(\sigma,\overline{X}(\sigma)),\overline{X}_{n}(\sigma)-\overline{X}(\sigma)\rangle_{V}\\
   &  \qquad +2\|B_{n}(\sigma,\overline{X}_{n}(\sigma))
      -B_{n}(\sigma,\overline{X}(\sigma))\|^{2}_{L_{2}(U,H)}\Big)\d\sigma\nonumber \\\nonumber
   &  \leq E\sup_{s\leq\tau\leq t}\int^{\tau}_{s}\left(\left|c\right|+L_{B}^{2}\right)\|\overline{X}_{n}(\sigma)
      -\overline{X}(\sigma)\|^{2}_{H}\d\sigma\\\nonumber
   &  \leq E\int^{t}_{s}\left(\left|c\right|+L_{B}^{2}\right)\|\overline{X}_{n}(\sigma)
      -\overline{X}(\sigma)\|^{2}_{H}\d\sigma\nonumber.
\end{align}
For the second term $\mathbb{I}_{2}$, by the H\"older inequality we obtain
\begin{align}\label{I2}
  \mathbb{I}_{2} & := E\sup_{s\leq\tau\leq t}\int^{\tau}_{s}\Big(2_{V^{*}}\langle A_{n}(\sigma,\overline{X}(\sigma))
     -A(\sigma,\overline{X}(\sigma)),\overline{X}_{n}(\sigma)-\overline{X}(\sigma)\rangle_{V}\\
   & \qquad +2\|B_{n}(\sigma,\overline{X}(\sigma))-B(\sigma,\overline{X}(\sigma))\|^{2}_{L_{2}(U,H)}\Big)\d\sigma\nonumber \\
   & \leq E\sup_{s\leq\tau\leq t}\int^{\tau}_{s}\Big(2\|A_{1,n}(\sigma,\overline{X}(\sigma))
     -A_{1}(\sigma,\overline{X}(\sigma))\|_{V_{1}^{*}}\|\overline{X}_{n}(\sigma)-\overline{X}(\sigma)\|_{V_{1}}\nonumber\\
   & \qquad +2\|A_{2,n}(\sigma,\overline{X}(\sigma))-A_{2}(\sigma,\overline{X}(\sigma))\|_{V_{2}^{*}}
     \|\overline{X}_{n}(\sigma)-\overline{X}(\sigma)\|_{V_{2}}\nonumber\\
   & \qquad +2\|B_{n}(\sigma,\overline{X}(\sigma))-B(\sigma,\overline{X}(\sigma))\|^{2}_{L_{2}(U,H)}\Big)\d\sigma\nonumber\\
   & \leq2\left(E\int^{t}_{s}\|A_{1,n}(\sigma,\overline{X}(\sigma))
     -A_{1}(\sigma,\overline{X}(\sigma))\|_{V_{1}^{*}}^{\frac{\alpha_{1}}{\alpha_{1}-1}}
     \d\sigma\right)^{\frac{\alpha_{1}-1}{\alpha_{1}}}
     \left(E\int^{t}_{s}\|\overline{X}_{n}(\sigma)-\overline{X}(\sigma)\|_{V_{1}}^{\alpha_{1}}
     \d\sigma\right)^{\frac{1}{\alpha_{1}}}\nonumber\\
   & \quad +2\left(E\int^{t}_{s}\|A_{2,n}(\sigma,\overline{X}(\sigma))
     -A_{2}(\sigma,\overline{X}(\sigma))\|_{V_{2}^{*}}^{\frac{\alpha_{2}}{\alpha_{2}-1}}
     \d \sigma\right)^{\frac{\alpha_{2}-1}{\alpha_{2}}}
     \left(E\int^{t}_{s}\|\overline{X}_{n}(\sigma)-\overline{X}(\sigma)\|_{V_{2}}^{\alpha_{2}}
     \d\sigma\right)^{\frac{1}{\alpha_{2}}}\nonumber\\
   & \quad +E\int^{t}_{s}2\|B_{n}(\sigma,\overline{X}(\sigma))
     -B(\sigma,\overline{X}(\sigma))\|^{2}_{L_{2}(U,H)}\d\sigma\nonumber.
\end{align}
For the last two terms $\mathbb{I}_{3}$ and $\mathbb{I}_{4}$, by Burkholder--Davis inequality
(see, e.g. \cite[Proposition D.0.1]{PR}) and Cauchy's inequality with $\epsilon$, we get
\begin{align}\label{I3}
  \mathbb{I}_{3} & := E\sup_{s\leq\tau\leq t}2\int^{\tau}_{s}\langle X_{n}(\sigma)-X(\sigma),
      \left[B_{n}(\sigma,\overline{X}_{n}(\sigma))-B_{n}(\sigma,\overline{X}(\sigma))\right]\d W(\sigma)\rangle_{H}\\
   &  \leq6E\left(\int^{t}_{s}\|B_{n}(\sigma,\overline{X}_{n}(\sigma))-B_{n}(\sigma,\overline{X}(\sigma))\|^{2}_{L_{2}(U,H)}
      \|X_{n}(\sigma)-X(\sigma)\|^{2}_{H}\d\sigma\right)^{1/2}\nonumber\\
   &  \leq C_{\epsilon}E\int^{t}_{s}L^{2}_{B}\|\overline{X}_{n}(\sigma)-\overline{X}(\sigma)\|^{2}_{H}\d\sigma
      +\epsilon E\sup_{s\leq\tau\leq t}\|X_{n}(\tau)-X(\tau)\|_{H}^{2}\nonumber
\end{align}
and
\begin{align}\label{I4}
  \mathbb{I}_{4} & := E\sup_{s\leq\tau\leq t}2\int^{\tau}_{s}\langle X_{n}(\sigma)-X(\sigma),
     \left[B_{n}(\sigma,\overline{X}(\sigma))-B(\sigma,\overline{X}(\sigma))\right]\d W(\sigma)\rangle_{H} \\
   & \leq6E\left(\int^{t}_{s}\|B_{n}(\sigma,\overline{X}(\sigma))-B(\sigma,\overline{X}(\sigma))\|^{2}_{L_{2}(U,H)}
     \|X_{n}(\sigma)-X(\sigma)\|^{2}_{H}\d\sigma\right)^{1/2}\nonumber\\
   & \leq C_{\epsilon}E\int^{t}_{s}\|B_{n}(\sigma,\overline{X}(\sigma))
     -B(\sigma,\overline{X}(\sigma))\|_{L_{2}(U,H)}^{2}\d\sigma
     +\epsilon E\sup_{s\leq\tau\leq t}\|X_{n}(\tau)-X(\tau)\|_{H}^{2}.\nonumber
\end{align}
Taking $\epsilon=\frac{1}{4}$  and combining \eqref{I1}--\eqref{I4}, we have
\begin{align*}
   & E\sup_{s\leq\tau\leq t}\|X_{n}(\tau)-X(\tau)\|^{2}_{H} \\
   & \leq 2E\|\zeta_{n}^{s}-\zeta^{s}\|^{2}_{H}
     +C_{1}E\int^{t}_{s}\sup_{s\leq u\leq\sigma}\|\overline{X}_{n}(u)-\overline{X}(u)\|^{2}_{H}\d\sigma\\
   & \quad +C_{2}\left(E\int^{t}_{s}\|A_{1,n}(\sigma,\overline{X}(\sigma))
     -A_{1}(\sigma,\overline{X}(\sigma))\|_{V_{1}^{*}}^{\frac{\alpha_{1}}{\alpha_{1}-1}}
     \d\sigma\right)^{\frac{\alpha_{1}-1}{\alpha_{1}}}
     \left(E\int^{t}_{s}\|\overline{X}_{n}(\sigma)-\overline{X}(\sigma)\|_{V_{1}}^{\alpha_{1}}
     \d\sigma\right)^{\frac{1}{\alpha_{1}}}\\
   & \quad +C_{2}\left(E\int^{t}_{s}\|A_{2,n}(\sigma,\overline{X}(\sigma))
     -A_{2}(\sigma,\overline{X}(\sigma))\|_{V_{2}^{*}}^{\frac{\alpha_{2}}{\alpha_{2}-1}}
     \d\sigma\right)^{\frac{\alpha_{2}-1}{\alpha_{2}}}
     \left(E\int^{t}_{s}\|\overline{X}_{n}(\sigma)-\overline{X}(\sigma)\|_{V_{2}}^{\alpha_{2}}
     \d\sigma\right)^{\frac{1}{\alpha_{2}}}\\
   & \quad +C_{3}E\int^{t}_{s}\|B_{n}(\sigma,\overline{X}(\sigma))
     -B(\sigma,\overline{X}(\sigma))\|^{2}_{L_{2}(U,H)}\d\sigma,
\end{align*}
where $C_{1}$, $C_{2}$ and $C_{3}$ are different positive constants, depending only on
$\epsilon$, $c$ and $L_{B}$. Then in view of the Gronwall's lemma, we have
\begin{equation}\label{conlemeq2}
  E\sup_{s\leq\tau\leq t}\|X_{n}(\tau)-X(\tau)\|^{2}_{H}\leq\xi_{n}{\rm{e}}^{C_{1}(t-s)},
\end{equation}
where
\begin{align*}
 \xi_{n} & := 2E\|\zeta_{n}^{s}-\zeta^{s}\|^{2}_{H}+C_{3}E\int^{t}_{s}\|B_{n}(\sigma,\overline{X}(\sigma))
     -B(\sigma,\overline{X}(\sigma))\|^{2}_{L_{2}(U,H)}\d\sigma \\
   & \quad +C_{2}\left(E\int^{t}_{s}\|A_{1,n}(\sigma,\overline{X}(\sigma))
     -A_{1}(\sigma,\overline{X}(\sigma))\|_{V_{1}^{*}}^{\frac{\alpha_{1}}{\alpha_{1}-1}}
     \d\sigma\right)^{\frac{\alpha_{1}-1}{\alpha_{1}}}
     \left(E\int^{t}_{s}\|\overline{X}_{n}(\sigma)-\overline{X}(\sigma)\|_{V_{1}}^{\alpha_{1}}
     \d\sigma\right)^{\frac{1}{\alpha_{1}}} \\
   & \quad +C_{2}\left(E\int^{t}_{s}\|A_{2,n}(\sigma,\overline{X}(\sigma))
     -A_{2}(\sigma,\overline{X}(\sigma))\|_{V_{2}^{*}}^{\frac{\alpha_{2}}{\alpha_{2}-1}}
     \d\sigma\right)^{\frac{\alpha_{2}-1}{\alpha_{2}}}
     \left(E\int^{t}_{s}\|\overline{X}_{n}(\sigma)-\overline{X}(\sigma)\|_{V_{2}}^{\alpha_{2}}
     \d\sigma\right)^{\frac{1}{\alpha_{2}}}.
  \end{align*}

Now, it suffices to prove that $\lim\limits_{n\rightarrow\infty}\xi_{n}=0$.
To this end, define
$$\gamma^{n}(R):=\inf\{\tau\geq s: \|X_{n}(\tau)\|_{H}\vee\|X(\tau)\|_{H}>R\}.$$
Using It\^o's formula and the product rule, we obtain
\begin{align*}
  & E\left({\rm{e}}^{-c_{1}(t-s)}\|X_{n}(t\wedge\gamma^{n}(R))\|_{H}^{2}\right) \\
  & =E\|\zeta_{n}^{s}\|_{H}^{2}+E\int_{s}^{t}{\rm{e}}^{-c_{1}(\sigma-s)}1_{[0,\gamma^{n}(R))}(\sigma)
     \Big(2_{V^{*}}\langle A_{n}(\sigma,\overline{X}_{n}(\sigma)),\overline{X}_{n}(\sigma) \rangle_{V}\\
  &  \qquad+\|B_{n}(\sigma,\overline{X}_{n}(\sigma))\|_{L_{2}(U,H)}^{2}\Big)\d\sigma
     +E\int_{s}^{t}-c_{1}{\rm{e}}^{-c_{1}(\sigma-s)}\|X_{n}(\sigma\wedge\gamma^{n}(R))\|_{H}^{2}\d\sigma.
\end{align*}
It follows from (H3) that
\begin{align}\label{conlemeq3}
 & E\left({\rm{e}}^{-c_{1}(t-s)}\|X_{n}(t\wedge\gamma^{n}(R))\|_{H}^{2}\right)+c_{1}E\int_{s}^{t}
     {\rm{e}}^{-c_{1}(\sigma-s)}\|X_{n}(\sigma\wedge\gamma^{n}(R))\|_{H}^{2}\d\sigma \\\nonumber
   & \quad +E\int_{s}^{t}{\rm{e}}^{-c_{1}(\sigma-s)}1_{[0,\gamma^{n}(R))}(\sigma)
           \left(c_{2}\|\overline{X}_{n}(\sigma)\|_{V_{1}}^{\alpha_{1}}
           +c'_{2}\|\overline{X}_{n}(\sigma)\|_{V_{2}}^{\alpha_{2}}\right)\d\sigma\\\nonumber
   & \leq E\|\zeta_{n}^{s}\|_{H}^{2}+E\int_{s}^{t}{\rm{e}}^{-c_{1}(\sigma-s)}M_{0}\d\sigma
         +E\int_{s}^{t}{\rm{e}}^{-c_{1}(\sigma-s)}c_{1}\|\overline{X}_{n}(\sigma)\|_{H}^{2}\d\sigma.\nonumber
\end{align}
Letting $R\rightarrow\infty$ in \eqref{conlemeq3} and using Fatou's lemma, we have
\begin{align}\label{conlemeq4}
   & E\left({\rm{e}}^{-c_{1}(t-s)}\|X_{n}(t)\|_{H}^{2}\right)+E\int_{s}^{t}{\rm{e}}^{-c_{1}(\sigma-s)}
     \left(c_{2}\|\overline{X}_{n}(\sigma)\|_{V_{1}}^{\alpha_{1}}
           +c'_{2}\|\overline{X}_{n}(\sigma)\|_{V_{2}}^{\alpha_{2}}\right)\d\sigma\\\nonumber
   & \leq E\|\zeta_{n}^{s}\|_{H}^{2}+E\int_{s}^{t}{\rm{e}}^{-c_{1}(\sigma-s)}M_{0}\d\sigma.\nonumber
\end{align}
Thus, by Lebesgue's dominated convergence theorem, (H4), Remark \ref{estB} (ii)
and \eqref{conlemeq4}, we obtain for $i=1,2$
$$\lim\limits_{n\rightarrow\infty}\left(E\int^{t}_{s}\|A_{i,n}(\sigma,\overline{X}(\sigma))
-A_{i}(\sigma,\overline{X}(\sigma))\|_{V_{i}^{*}}^{\frac{\alpha_{i}}{\alpha_{i}-1}}
  \d\sigma\right)^{\frac{\alpha_{i}-1}{\alpha_{i}}}
  \left(E\int^{t}_{s}\|\overline{X}_{n}(\sigma)-\overline{X}(\sigma)\|_{V_{i}}^{\alpha_{i}}
  \d\sigma\right)^{\frac{1}{\alpha_{i}}}=0,$$
$$\lim\limits_{n\rightarrow\infty}E\int^{t}_{s}\|B_{n}(\sigma,\overline{X}(\sigma))
  -B(\sigma,\overline{X}(\sigma))\|_{L_{2}(U,H)}^{2}\d\sigma=0.$$
Therefore,
\[
\lim\limits_{n\rightarrow\infty}\xi_{n}=0.
\]
The proof of (i) is complete.

(ii) According to the characterization of convergence in probability in terms of $P$-a.s.
convergent subsequences (see, e.g. \cite[Theorem 9.2.1]{Dudley}),
we may assume without loss of generality that
$\lim\limits_{n\rightarrow\infty}\zeta_{n}^{s}=\zeta^{s}$ $P$-a.s.
Similar to the proof of \eqref{conlemeq4}, we have
\begin{align}\label{bd}
   & E\left({\rm{e}}^{-c_{1}(\tau-s)-\sup\limits_{n}\|\zeta^{s}_{n}\|_{H}}\|X_{n}(\tau)\|_{H}^{2}\right)\\\nonumber
   &  \quad+E\int_{s}^{\tau}{\rm{e}}^{-c_{1}(\sigma-s)-\sup\limits_{n}\|\zeta^{s}_{n}\|_{H}}
     \left(c_{2}\|\overline{X}_{n}(\sigma)\|_{V_{1}}^{\alpha_{1}}
     +c'_{2}\|\overline{X}_{n}(\sigma)\|_{V_{2}}^{\alpha_{2}}\right)\d\sigma \\\nonumber
   & \leq E\left({\rm{e}}^{-\sup\limits_{n}\|\zeta^{s}_{n}\|_{H}}\|\zeta^{s}_{n}\|_{H}^{2}\right)
     +E\int_{s}^{t}{\rm{e}}^{-c_{1}(\sigma-s)-\sup\limits_{n}\|\zeta^{s}_{n}\|_{H}}M_{0}\d\sigma.
\end{align}

Applying It\^o's formula and the product rule, we get
\begin{align*}
   & \|X_{n}(\tau)-X(\tau)\|_{H}^{2}{\rm{e}}^{-c(\tau-s)-\sup\limits_{n}\|\zeta^{s}_{n}\|_{H}} \\
   & =\|\zeta_{n}^{s}-\zeta^{s}\|_{H}^{2}{\rm{e}}^{-\sup\limits_{n}\|\zeta^{s}_{n}\|_{H}}
      +\int_{s}^{\tau}-c{\rm{e}}^{-c(\sigma-s)-\sup\limits_{n}\|\zeta^{s}_{n}\|_{H}}\|X_{n}(\sigma)
      -X(\sigma)\|_{H}^{2}\d\sigma\\
   & \quad +\int_{s}^{\tau}{\rm{e}}^{-c(\sigma-s)-\sup\limits_{n}\|\zeta^{s}_{n}\|_{H}}
     \Big(2_{V^{*}}\langle A_{n}(\sigma,\overline{X}_{n}(\sigma))-A(\sigma,\overline{X}(\sigma)),
      \overline{X}_{n}(\sigma)-\overline{X}(\sigma)\rangle_{V}\\
   & \qquad +\|B_{n}(\sigma,\overline{X}_{n}(\sigma))-B(\sigma,\overline{X}(\sigma))\|_{L_{2}(U,H)}^{2}\Big)\d\sigma\\
   &\quad +2\int_{s}^{\tau}{\rm{e}}^{-c(\sigma-s)-\sup\limits_{n}\|\zeta^{s}_{n}\|_{H}}\langle X_{n}(\sigma)-X(\sigma),
     \left[B_{n}(\sigma,\overline{X}_{n}(\sigma))-B(\sigma,\overline{X}(\sigma))\right]\d W(\sigma)\rangle_{H}.
\end{align*}
Note that the last item is a real-valued local martingale. Hence localizing it,
by Lebesgue's dominated convergence theorem and (H2) we obtain that
\begin{align*}
   & E\left(\|X_{n}(\tau)-X(\tau)\|_{H}^{2}{\rm{e}}^{-c(\tau-s)-\sup\limits_{n}\|\zeta^{s}_{n}\|_{H}}\right) \\
   & \leq E\left(\|\zeta_{n}^{s}-\zeta^{s}\|_{H}^{2}{\rm{e}}^{-\sup\limits_{n}\|\zeta^{s}_{n}\|_{H}}\right)
     +E\int_{s}^{\tau}-c{\rm{e}}^{-c(\sigma-s)-\sup\limits_{n}\|\zeta^{s}_{n}\|_{H}}
     \|X_{n}(\sigma)-X(\sigma)\|_{H}^{2}\d\sigma\\
   & \quad + E\int_{s}^{\tau}{\rm{e}}^{-c(\sigma-s)-\sup\limits_{n}\|\zeta^{s}_{n}\|_{H}}
     \Big(2_{V^{*}}\langle A_{n}(\sigma,\overline{X}_{n}(\sigma))-A_{n}(\sigma,\overline{X}(\sigma)),
     \overline{X}_{n}(\sigma)-\overline{X}(\sigma)\rangle_{V}\\
   & \qquad +2\|B_{n}(\sigma,\overline{X}_{n}(\sigma))-B_{n}(\sigma,\overline{X}(\sigma))\|_{L_{2}(U,H)}^{2}
     +2\|B_{n}(\sigma,\overline{X}(\sigma))-B(\sigma,\overline{X}(\sigma))\|_{L_{2}(U,H)}^{2}\\
   & \qquad +2_{V^{*}}\langle A_{n}(\sigma,\overline{X}(\sigma))-A(\sigma,\overline{X}(\sigma)),
     \overline{X}_{n}(\sigma)-\overline{X}(\sigma)\rangle_{V}\Big)\d\sigma \\
   & \leq E\left(\|\zeta_{n}^{s}-\zeta^{s}\|_{H}^{2}{\rm{e}}^{-\sup\limits_{n}\|\zeta^{s}_{n}\|_{H}}\right)
     +E\int_{s}^{\tau}{\rm{e}}^{-c(\sigma-s)-\sup\limits_{n}\|\zeta^{s}_{n}\|_{H}}
     \Big(L_{B}^{2}\|\overline{X}_{n}(\sigma)-\overline{X}(\sigma)\|_{H}^{2}\\
   & \qquad+2\|A_{1,n}(\sigma,\overline{X}(\sigma))-A_{1}(\sigma,\overline{X}(\sigma))\|_{V_{1}^{*}}
     \|\overline{X}_{n}(\sigma)-\overline{X}(\sigma)\|_{V_{1}}\\
   & \qquad +2\|A_{2,n}(\sigma,\overline{X}(\sigma))-A_{2}(\sigma,\overline{X}(\sigma))\|_{V_{2}^{*}}
     \|\overline{X}_{n}(\sigma)-\overline{X}(\sigma)\|_{V_{2}}\\
   & \qquad +2\|B_{n}(\sigma,\overline{X}(\sigma))-B(\sigma,\overline{X}(\sigma))\|_{L_{2}(U,H)}^{2}\Big)\d\sigma.
\end{align*}
Therefore, in view of Gronwall's lemma, we get
\begin{equation}\label{contii2}
  E\left(\|X_{n}(\tau)-X(\tau)\|_{H}^{2}{\rm{e}}^{-c(\tau-s)-\sup\limits_{n}\|\zeta^{s}_{n}\|_{H}}\right)
  \leq \tilde\xi_{n}{\rm{e}}^{L_{B}^{2}(t-s)}, \quad \rm{for~all}~ \tau\in[s,t],
\end{equation}
where
\begin{align*}
 \tilde\xi_n:=&E\left(\|\zeta_{n}^{s}-\zeta^{s}\|_{H}^{2}{\rm{e}}^{-\sup\limits_{n}\|\zeta^{s}_{n}\|_{H}}\right)\\
    & +2\left(E\int_{s}^{t}\|A_{1,n}(\sigma,\overline{X}(\sigma))
    -A_{1}(\sigma,\overline{X}(\sigma))\|_{V_{1}^{*}}^{\frac{\alpha_{1}}{\alpha_{1}-1}}
    \d\sigma\right)^{\frac{\alpha_{1}-1}{\alpha_{1}}}\\
    &\qquad \times
    \left(E\int_{s}^{t}{\rm{e}}^{-c(\sigma-s)-\sup\limits_{n}\|\zeta^{s}_{n}\|_{H}}
     \|\overline{X}_{n}(\sigma)-\overline{X}(\sigma)\|_{V_{1}}^{\alpha_{1}}\d\sigma\right)^{\frac{1}{\alpha_{1}}}\\
  &  +2\left(E\int_{s}^{t}\|A_{2,n}(\sigma,\overline{X}(\sigma))
    -A_{2}(\sigma,\overline{X}(\sigma))\|_{V_{2}^{*}}^{\frac{\alpha_{2}}{\alpha_{2}-1}}
    \d\sigma\right)^{\frac{\alpha_{2}-1}{\alpha_{2}}}\\
    &\qquad \times
    \left(E\int_{s}^{t}{\rm{e}}^{-c(\sigma-s)-\sup\limits_{n}\|\zeta^{s}_{n}\|_{H}}
    \|\overline{X}_{n}(\sigma)-\overline{X}(\sigma)\|_{V_{2}}^{\alpha_{2}}
    \d\sigma\right)^{\frac{1}{\alpha_{2}}}\\
  &  +2E\int_{s}^{t}{\rm{e}}^{-c(\sigma-s)-\sup\limits_{n}\|\zeta^{s}_{n}\|_{H}}
    \|B_{n}(\sigma,\overline{X}(\sigma))-B(\sigma,\overline{X}(\sigma))\|_{L_{2}(U,H)}^{2}\d\sigma.
\end{align*}
By Lebesgue's dominated convergence theorem, \eqref{bd}, (H4) and Remark \ref{estB} (i), we have
\begin{equation}\label{contii3}
  \lim_{n\rightarrow\infty} \tilde\xi_{n}=0.
\end{equation}

For any $\epsilon>0$, let
$$\tau_{\epsilon}^n:=\inf\{\sigma\geq s :  \|X_{n}(\sigma)-X(\sigma)\|_{H}^{2}{\rm{e}}^{-c(\sigma-s)
  -\sup\limits_{n}\|\zeta^{s}_{n}\|_{H}}\geq\epsilon\}\wedge t.$$
It follows from \cite[Lemma 3.1.3]{PR} and \eqref{contii2} that
\begin{align*}
   & P\left(\sup\limits_{s\leq \sigma\leq t}\|X_{n}(\sigma)-X(\sigma)\|_{H}^{2}
     {\rm{e}}^{-c\left(\sigma-s\right)-\sup\limits_{n}\|\zeta^{s}_{n}\|_{H}}\geq\epsilon\right) \\
   & \leq\frac{1}{\epsilon}E\left(\|X_{n}(\tau_{\epsilon}^n)-X(\tau_{\epsilon}^n)\|_{H}^{2}
     {\rm{e}}^{-c\left(\tau_{\epsilon}^n-s\right)-\sup\limits_{n}\|\zeta^{s}_{n}\|_{H}}\right)\\
   & \le \frac{1}{\epsilon} \tilde \xi_{n}{\rm{e}}^{L_{B}^{2}(t-s)}.
\end{align*}
So by \eqref{contii3} we have
\begin{equation}\label{p0}
 \lim\limits_{n\rightarrow\infty}P\left(\sup\limits_{s\leq \sigma\leq t}
 \|X_{n}(\sigma)-X(\sigma)\|_{H}^{2}{\rm{e}}^{-c\left(\sigma-s\right)-\sup\limits_{n}
 \|\zeta^{s}_{n}\|_{H}}\geq\epsilon\right)=0 .
\end{equation}
Since $P$-a.s. $[0,\infty)\ni t\mapsto {\rm{e}}^{-c\left(t-s\right)-\sup\limits_{n}\|\zeta^{s}_{n}\|_{H}}$
is continuous and strictly positive, \eqref{p0} implies
\begin{equation*}
\lim\limits_{n\rightarrow\infty}\sup\limits_{s\leq\sigma\leq t}\|X_{n}(\sigma)-X(\sigma)\|_{H}^{2}=0
\quad \rm in ~probability.
\end{equation*}
This completes the proof of (ii).

(iii) According to the Skorohod representation theorem, the uniqueness in law of the solutions
for equation \eqref{eqSPDE1} and (ii), we complete the proof of (iii).

\end{proof}

\begin{remark}\rm\label{Remc}
Note that $\lim\limits_{n\rightarrow\infty}d_{BL}(\mathcal{L}(X_{n}),\mathcal{L}(X))=0$ in $Pr(C(\R,H))$
implies
\[
\lim\limits_{n\rightarrow\infty} d_{BL} (\mathcal{L}(X_{n}(t)),\mathcal{L}(X(t))) = 0 \quad\hbox{in } Pr(H) \hbox{ for each }t\in\R,
\]
but not vice versa. However, under the conditions of Theorem \ref{conlemma}, it follows from (iii) that
$\lim\limits_{n\rightarrow\infty}d_{BL}(\mathcal{L}(X_{n}(t)),\mathcal{L}(X(t)))=0$ in $Pr(H)$ for $t\in\R$
if and only if
$\lim\limits_{n\rightarrow\infty}d_{BL}(\mathcal{L}(X_{n}),\mathcal{L}(X))=0$ in $Pr(C(\R,H))$ ---
in other words $X_n$ converges in distribution to $X$ on the path space $C(\R,H)$ if and only if $X_n$
converges in distribution to $X$ on $H$.
Indeed, if for any $s\in\R$ $\lim\limits_{n\rightarrow\infty}d_{BL}(\mathcal{L}(X_{n}(s)),\mathcal{L}(X(s)))=0$ in $Pr(H)$,
then by Theorem \ref{conlemma}, for any interval $[s,t]\subset\R$ we have
$$\lim\limits_{n\rightarrow\infty}d_{BL}(\mathcal{L}(X_{n}),\mathcal{L}(X))=0
\quad {\rm{in}} ~ Pr(C([s,t],H)).$$
It follows from the arbitrariness of $[s,t]$ that
$$\lim_{n\rightarrow\infty}d_{BL}(\mathcal{L}(X_{n}),\mathcal{L}(X))=0
\quad {\rm{in}}~ Pr(C(\mathbb{R},H)).$$
\end{remark}

Now we discuss the $L^{2}$-bounded solution to equation \eqref{eqSPDE1}. In the following,
we need the stronger condition:

(H2$'$) (Strict monotonicity) There exists a constant $\lambda>0$ such that for all
$u$, $v\in V$, $t\in \mathbb{R}$
\begin{align*}
  2_{V^{*}}\langle A(t,u)-A(t,v),u-v\rangle_{V}+\|B(t,u)-B(t,v)\|^{2}_{L_{2}(U,H)}  \leq -\lambda\|u-v\|^{2}_{H}.
\end{align*}

\begin{lemma}
Suppose that {\rm{(H2$'$)}}, {\rm{(H3)}} and {\rm{(H4)}} hold. Let $\eta\in(0,\lambda)$.
Then there exists a constant $M_{0,\eta}\in(0,\infty)$, depending only on $\eta$, $c_{2}$,
$c_{3}$, $c_{2}'$, $c_{3}'$, $M_{0}$, $\alpha_{i}$, $i=1,2$ such that
\begin{equation}\label{ineq1}
2_{V^{*}}\langle A(t,v),v\rangle_{V}+\|B(t,v)\|^{2}_{L_{2}(U,H)}\leq -\eta\|v\|^{2}_{H}+M_{0,\eta}
\end{equation}
for all $v\in V$, $t\in \mathbb{R}$.
\end{lemma}

\begin{proof}
Let $v\in V$ and $\epsilon\in(0,1)$. Employing Young's inequality, (H2$'$), (H3) and (H4), we obtain
\begin{align*}
   & 2_{V^{*}}\langle A(t,v),v\rangle_{V}+\|B(t,v)\|^{2}_{L_{2}(U,H)} \\
   & \leq2_{V^{*}}\langle A(t,v)-A(t,0),v\rangle_{V}+2_{V^{*}}\langle A(t,0),v\rangle_{V}
     +\|B(t,v)-B(t,0)\|^{2}_{L_{2}(U,H)}\\
   & \quad -\|B(t,0)\|^{2}_{L_{2}(U,H)}+2\langle B(t,v),B(t,0)\rangle_{L_{2}(U,H)}\\
   & \leq-\lambda\|v\|^{2}_{H}+2_{V_{1}^{*}}\langle A_{1}(t,0),v\rangle_{V_{1}}
     +2_{V_{2}^{*}}\langle A_{2}(t,0),v\rangle_{V_{2}}\\
   & \quad+2\|B(t,v)\|_{L_{2}(U,H)}\|B(t,0)\|_{L_{2}(U,H)}-\|B(t,0)\|^{2}_{L_{2}(U,H)}\\
   & \leq-\lambda\|v\|^{2}_{H}+2\|A_{1}(t,0)\|_{V_{1}^{*}}\|v\|_{V_{1}}
     +2\|A_{2}(t,0)\|_{V_{2}^{*}}\|v\|_{V_{2}}+2\|B(t,v)\|_{L_{2}(U,H)}\|B(t,0)\|_{L_{2}(U,H)}\\
   & \leq-\lambda\|v\|^{2}_{H}+\epsilon\|v\|_{V_{1}}^{\alpha_{1}}+C_{\epsilon}^{1}
     \|A_{1}(t,0)\|_{V_{1}^{*}}^{\frac{\alpha_{1}}{\alpha_{1}-1}}+\epsilon\|v\|_{V_{2}}^{\alpha_{2}}\\
   & \quad +C_{\epsilon}^{2}\|A_{2}(t,0)\|_{V_{2}^{*}}^{\frac{\alpha_{2}}{\alpha_{2}-1}}
      +\epsilon\|B(t,v)\|_{L_{2}(U,H)}^{2}+C_{\epsilon}^{3}\|B(t,0)\|_{L_{2}(U,H)}^{2}\\
   & \leq-\lambda\|v\|^{2}_{H}+\epsilon\|v\|_{V_{1}}^{\alpha_{1}}
     +C_{\epsilon}^{1}M_{0}^{\frac{\alpha_{1}}{\alpha_{1}-1}}
     +\epsilon\|v\|_{V_{2}}^{\alpha_{2}}+C_{\epsilon}^{2}M_{0}^{\frac{\alpha_{2}}{\alpha_{2}-1}}
     +\epsilon c_{1}\|v\|^{2}_{H}+\epsilon\left(2c_{3}-c_{2}
     +\frac{2}{\alpha_{1}}\right)\|v\|^{\alpha_{1}}_{V_{1}}\\
   & \quad  +\epsilon\left(2c_{3}'-c'_{2}+\frac{2}{\alpha_{2}}\right)\|v\|^{\alpha_{2}}_{V_{2}}+\epsilon M_{0}
      +\epsilon\frac{2\left(\alpha_{1}-1\right)}{\alpha_{1}}M_{0}^{\frac{\alpha_{1}}{\alpha_{1}-1}}
       +\epsilon\frac{2\left(\alpha_{2}-1\right)}{\alpha_{2}}M_{0}^{\frac{\alpha_{2}}{\alpha_{2}-1}}+C_{\epsilon}^{3}M_{0}\\
   & \leq\left(-\lambda+\epsilon c_{1}\right)\|v\|^{2}_{H}
     +\epsilon\left(1+2c_{3}-c_{2}+\frac{2}{\alpha_{1}}\right)\|v\|_{V_{1}}^{\alpha_{1}}
       +\epsilon\left(1+2c_{3}'-c_{2}'+\frac{2}{\alpha_{2}}\right)\|v\|_{V_{2}}^{\alpha_{2}}+\widetilde{M_{0}}\\
   & \leq\left(-\lambda+\epsilon c_{1}\right)\|v\|^{2}_{H}+\epsilon\left(1+2c_{3}-c_{2}+\frac{2}{\alpha_{1}}\right)
     \bigg[-\frac{1}{c_{2}}\left(2_{V^{*}}\langle A(t,v),v\rangle_{V}+\|B(t,v)\|^{2}_{L_{2}(U,H)}\right)\\
   & \qquad +\frac{c_{1}}{c_{2}}\|v\|^{2}_{H}+\frac{1}{c_{2}}M_{0}\bigg]+\epsilon\left(1+2c_{3}'-c_{2}'
     +\frac{2}{\alpha_{2}}\right)\bigg[-\frac{1}{c_{2}'}\Big(2_{V^{*}}\langle A(t,v),v\rangle_{V}\\
   & \qquad +\|B(t,v)\|^{2}_{L_{2}(U,H)}\Big)+\frac{c_{1}}{c_{2}'}\|v\|^{2}_{H}
     +\frac{1}{c_{2}'}M_{0}\bigg]+\widetilde{M_{0}}\\
   & \leq\left(-\lambda+\epsilon c_{1}+\epsilon\left(1+2c_{3}-c_{2}
     +\frac{2}{\alpha_{1}}\right)\frac{c_{1}}{c_{2}}+
     \epsilon\left(1+2c_{3}'-c_{2}'+\frac{2}{\alpha_{2}}\right)\frac{c_{1}}{c_{2}'}\right)\|v\|_{H}^{2}\\
   & \quad -\left[\frac{\epsilon}{c_{2}}\left(1+2c_{3}-c_{2}+\frac{2}{\alpha_{1}}\right)
     +\frac{\epsilon}{c_{2}'}\left(1+2c_{3}'-c_{2}'+\frac{2}{\alpha_{2}}\right)\right]\Big(2_{V^{*}}\langle A(t,v),v\rangle_{V}\\
   & \qquad +\|B(t,v)\|^{2}_{L_{2}(U,H)}\Big)+\widetilde{M_{0}}',
\end{align*}
where $C_{\epsilon}^{1}$, $C_{\epsilon}^{2}$, $C_{\epsilon}^{3}$, $\widetilde{M_{0}}$, $\widetilde{M_{0}}'\in(0,\infty)$
are constants independent of $v$ and $t$.
Hence taking $\epsilon$ small enough we can find a constant $M_{0,\eta}$
such that for all $v\in V$, $t\in \mathbb{R}$
$$2_{V^{*}}\langle A(t,v),v\rangle_{V}+\|B(t,v)\|^{2}_{L_{2}(U,H)}\leq -\eta\|v\|^{2}_{H}+M_{0,\eta}.$$
\end{proof}

\begin{lemma}
Assume that {\rm{(H1)}}, {\rm{(H2$'$)}}, {\rm{(H3)}} and {\rm{(H4)}} hold.
Let $\zeta_{s}\in L^{2}(\Omega, \mathcal{F}_{s}, P;H)$ and
$X(t,s,\zeta_{s})$, $t\geq s$ be the solution to the following Cauchy problem
\begin{equation*}
  \left\{
   \begin{aligned}
   &\ \d X(t)=A(t,X(t))\d t+B(t,X(t))\d W(t)\\
  &\ X(s)=\zeta_{s}.
   \end{aligned}
   \right.
  \end{equation*}
Then there exists a constant $M_{1}>0$,
depending only on $M_{0,\eta}$ as in \eqref{ineq1}, such that
\begin{equation}\label{ineq2}
E\|X(t,s,\zeta_{s})\|^{2}_{H}\leq {\rm{e}}^{-\eta(t-s)}E\|\zeta_{s}\|^{2}_{H}+M_{1}.
\end{equation}
Furthermore, for some process $X(t)$, $t\in\R$, we have
\begin{equation}\label{L2lim1}
  X(t,-n,0)\rightarrow X(t)\quad {\rm{in}}~ L^{2}(\Omega,P;H).
\end{equation}
\end{lemma}

\begin{proof}
By the product rule, It\^o's formula and \eqref{ineq1}, we have
\begin{align*}
& E\left({\rm{e}}^{\eta(t-s)}\|X(t,s,\zeta_{s})\|^{2}_{H}\right)\\
& =E\|\zeta_{s}\|^{2}_{H}+\int^{t}_{s}\eta {\rm{e}}^{\eta(\sigma-s)}E\|X(\sigma,s,\zeta_{s})\|^{2}_{H}\d\sigma\\
& \quad +E\int^{t}_{s}{\rm{e}}^{\eta(\sigma-s)}\left(2_{V^{*}}\langle A(\sigma,\overline{X}(\sigma,s,\zeta_{s})),
  \overline{X}(\sigma,s,\zeta_{s})\rangle_{V}
  +\|B(\sigma,\overline{X}(\sigma,s,\zeta_{s}))\|^{2}_{L_{2}(U,H)}\right)\d\sigma\\
& \leq E\|\zeta_{s}\|^{2}_{H}+\int^{t}_{s}{\rm{e}}^{\eta(\sigma-s)}M_{0,\eta}\d\sigma\\
& \leq E\|\zeta_{s}\|^{2}_{H}+\frac{M_{0,\eta}}{\eta}{\rm{e}}^{\eta(t-s)}.
\end{align*}
Let $M_{1}:=\frac{M_{0,\eta}}{\eta}$, we obtain
$$E\|X(t,s,\zeta_{s})\|^{2}_{H}\leq {\rm{e}}^{-\eta(t-s)}E\|\zeta_{s}\|^{2}_{H}+M_{1}. $$

For $t\geq-m\geq-n$, in view of It\^o's formula, we have
\begin{align*}
& E\|X(t,-n,0)-X(t,-m,0)\|^{2}_{H}\\
& =E\|X(-m,-n,0)\|^{2}_{H}\\
& \quad +E\int^{t}_{-m}\Big(2_{V^{*}}\langle A(\sigma,\overline{X}(\sigma,-n,0))
   -A(\sigma,\overline{X}(\sigma,-m,0)),\overline{X}(\sigma,-n,0)-\overline{X}(\sigma,-m,0)\rangle_{V}\\
& \qquad +\|B(\sigma,\overline{X}(\sigma,-n,0))-B(\sigma,\overline{X}(\sigma,-m,0))\|^{2}_{L_{2}(U,H)}\Big)\d\sigma.
\end{align*}
By (H2$'$) and the product rule, we obtain
\begin{align*}
   &  E\left({\rm{e}}^{\lambda(t+m)}\|X(t,-n,0)-X(t,-m,0)\|^{2}_{H}\right)  \\
   & =E\|X(-m,-n,0)\|^{2}_{H}+\int^{t}_{-m}\lambda {\rm{e}}^{\lambda(\sigma+m)}E\|X(\sigma,-n,0)
     -X(\sigma,-m,0)\|^{2}_{H}\d\sigma\\
   & \quad +E\int^{t}_{-m}{\rm{e}}^{\lambda(\sigma+m)}\Big(2_{V^{*}}\langle A(\sigma,\overline{X}(\sigma,-n,0))
     -A(\sigma,\overline{X}(\sigma,-m,0)),\overline{X}(\sigma,-n,0)-\overline{X}(\sigma,-m,0)\rangle_{V}\\
   & \qquad +\|B(\sigma,\overline{X}(\sigma,-n,0))-B(\sigma,\overline{X}(\sigma,-m,0))\|^{2}_{L_{2}(U,H)}\Big)\d\sigma\\
   & \leq E\|X(-m,-n,0)\|^{2}_{H}.
\end{align*}
Now using  \eqref{ineq2} we deduce
 \begin{equation*}
   E\left(\|X(t,-n,0)-X(t,-m,0)\|^{2}_{H}\right)\leq E\|X(-m,-n,0)\|^{2}_{H}
   {\rm{e}}^{-\lambda(t+m)}\leq M_{1}{\rm{e}}^{-\lambda(t+m)}.
 \end{equation*}
Letting $n>m$, $m\rightarrow\infty$, we have
$$E\left(\|X(t,-n,0)-X(t,-m,0)\|^{2}_{H}\right)\rightarrow 0.$$
Therefore, there exists a process $X(t)$, $t\in\R$ such that
$$X(t,-n,0)\rightarrow X(t) \quad {\rm {in}}~L^{2}(\Omega,P;H).$$
And it follows from \eqref{ineq2} that $\sup\limits_{t\in\mathbb{R}}E\|X(t)\|_{H}^{2}\leq M_{1}$.
\end{proof}

We now show that the limit process $X(\cdot)$ in \eqref{L2lim1} is a solution to equation \eqref{eqSPDE1}.
For this we need some uniform estimates.

\begin{lemma}\label{Bestimate}
Consider equation \eqref{eqSPDE1}.
Assume that {\rm{(H1)}}, {\rm{(H2$'$)}}, {\rm{(H3)}} and {\rm{(H4)}} hold.
For any fixed interval $[a,b]\subset \mathbb{R}$,
there exists a constant $M_{2}$, depending only on $M_{1}$, $M_{0}$, $c_{1}$ and $[a,b]$, such that
$$\sup_{t\in\R}E\|X(t,-n,0)\|^{2}_{H}+\sum_{i=1,2}\|\overline{X}(\cdot,-n,0)\|_{K_{i}}
+\sum_{i=1,2}\|A_{i}(\cdot,\overline{X}(\cdot,-n,0))\|_{K_{i}^{*}}\leq M_{2}$$
for all $-n\leq a$, where
$K_{i}:=L^{\alpha_{i}}([a,b]\times\Omega,\d t\otimes P;V_{i}),$
$K_{i}^{*}:=L^{\frac{\alpha_{i}}{\alpha_{i}-1}}([a,b]\times\Omega,\d t\otimes P;V_{i}^{*})$, $i=1,2.$
\end{lemma}

\begin{proof}
According to (H3), we have
\begin{align*}
   & E\left({\rm{e}}^{-c_{1}(t-a)}\|X(t,-n,0)\|^{2}_{H}\right) \\
   & = E\|X(a,-n,0)\|^{2}_{H}-\int^{t}_{a}c_{1}{\rm{e}}^{-c_{1}(\sigma-a)}E\|X(\sigma,-n,0)\|^{2}_{H}\d\sigma\\
   & \quad +\int^{t}_{a}{\rm{e}}^{-c_{1}(\sigma-a)}E\left(2_{V^{*}}\langle A(\sigma,\overline{X}(\sigma,-n,0)),
     \overline{X}(\sigma,-n,0)\rangle_{V}+
     \|B(\sigma,\overline{X}(\sigma,-n,0))\|^{2}_{L_{2}(U,H)}\right)\d\sigma \\
   & \leq E\|X(a,-n,0)\|^{2}_{H}
     -c_{2}\int^{t}_{a}{\rm{e}}^{-c_{1}(\sigma-a)}E\|\overline{X}(\sigma,-n,0)\|^{\alpha_{1}}_{V_{1}}\d\sigma\\
   & \quad -c_{2}'\int^{t}_{a}{\rm{e}}^{-c_{1}(\sigma-a)}E\|\overline{X}(\sigma,-n,0)\|^{\alpha_{2}}_{V_{2}}\d\sigma
     +\int^{t}_{a}{\rm{e}}^{-c_{1}(\sigma-a)}M_{0}\d\sigma. \\
\end{align*}
Therefore,
\begin{align*}
   & E\left({\rm{e}}^{-c_{1}(t-a)}\|X(t,-n,0)\|^{2}_{H}\right)+c_{2}\int^{t}_{a}{\rm{e}}^{-c_{1}(\sigma-a)}
     E\|\overline{X}(\sigma,-n,0)\|^{\alpha_{1}}_{V_{1}}\d\sigma\\
   & \quad +c_{2}'\int^{t}_{a}{\rm{e}}^{-c_{1}(\sigma-a)}E\|\overline{X}(\sigma,-n,0)\|^{\alpha_{2}}_{V_{2}}\d\sigma \\
   & \leq E\|X(a,-n,0)\|^{2}_{H}+\int^{t}_{a}{\rm{e}}^{-c_{1}(\sigma-a)}M_{0}\d\sigma.
\end{align*}
In view of (H4), we complete the proof.
\end{proof}

\begin{theorem}\label{Boundedth}
Consider equation \eqref{eqSPDE1}.
Suppose that {\rm{(H1)}}, {\rm{(H2$'$)}}, {\rm{(H3)}} and {\rm{(H4)}} hold, then there exists a
unique $L^{2}$-bounded continuous H-valued solution $X(t)$, $t\in\mathbb{R}$ to
equation \eqref{eqSPDE1}. Moreover, the mapping $\widehat{\mu}:\mathbb{R}\rightarrow Pr(H)$,
defined by $\widehat{\mu}(t):=P\circ[X(t)]^{-1}$, is unique with the following properties:
\begin{enumerate}
\item $L^{2}$-boundedness: ~$\sup\limits_{t\in \mathbb{R}}\int_{H}\|x\|^{2}_{H}\widehat{\mu}(t)(\d x)<+\infty$;
\item Flow property: ~$\mu(t,s,\widehat{\mu}(s))=\widehat{\mu}(t)$ for all $t\geq s$.
\end{enumerate}
Here $\mu(t,s,\mu_{0})$ denotes the distribution of $X(t,s,\zeta_{s})$ on $H$,
with $\mu_{0}=P\circ\zeta^{-1}_{s}$.
\end{theorem}

\begin{proof}
For any fixed interval $[a,b]\subset \mathbb{R}$, we denote
$$J:=L^{2}([a,b]\times\Omega,\d t\otimes P;L_{2}(U,H)),
\quad K_{i}:=L^{\alpha_{i}}([a,b]\times\Omega,\d t\otimes P;V_{i}),$$
$$K_{i}^{*}:=L^{\frac{\alpha_{i}}{\alpha_{i}-1}}([a,b]\times\Omega,\d t\otimes P;V_{i}^{*}),~i=1,2.$$
According to the reflexivity of $K_{i}$, $i=1,2$, we may assume, going if necessary to a subsequence, that
\begin{enumerate}
\item[(1)] $X(\cdot,-n,0)\rightarrow X(\cdot)$ in $L^{2}([a,b]\times\Omega,\d t\otimes P;H)$ and
      $\overline{X}(\cdot,-n,0)\rightarrow \overline{X}(\cdot) $ weakly in $K_{1}$ and $K_{2}$;
\item[(2)] $A_{i}(\cdot,\overline{X}(\cdot,-n,0))\rightarrow Y_{i}(\cdot)$ weakly in $K_{i}^{*}$, $i=1,2$;
\item[(3)] $B(\cdot,\overline{X}(\cdot,-n,0))\rightarrow Z(\cdot)$ weakly in $J$ and hence
      $$\int^{t}_{a}B(\sigma,\overline{X}(\sigma,-n,0))\d W(\sigma)\rightarrow\int^{t}_{a}Z(\sigma)\d W(\sigma)$$
      weakly* in $L^{\infty}([a,b],\d t;L^{2}(\Omega,P;H))$.
\end{enumerate}

Thus for all $v\in V$, $\varphi\in L^{\infty}([a,b]\times\Omega)$ by Fubini's theorem we get
\begin{align*}
   & E\int^{b}_{a}~_{V^{*}}\langle X(t),\varphi(t)v\rangle_{V}\d t \\
   & =\lim_{n\rightarrow\infty}E\int^{b}_{a}~_{V^{*}}\langle X(t,-n,0),\varphi(t)v\rangle_{V}\d t \\
   & =\lim_{n\rightarrow\infty}E\int^{b}_{a}~_{V^{*}}\langle X(a,-n,0),\varphi(t)v\rangle_{V}\d t
     +\lim_{n\rightarrow\infty}E\int^{b}_{a}\int^{b}_{\sigma}~_{V^{*}}
     \langle A(\sigma,\overline{X}(\sigma,-n,0)),\varphi(t)v\rangle_{V}\d t\d\sigma\\
   & \quad + \lim_{n\rightarrow\infty}E\left(\int^{b}_{a}\langle \int^{t}_{a}
     B(\sigma,\overline{X}(\sigma,-n,0))\d W(\sigma),\varphi(t)v\rangle_{H}\d t\right)\\
   & =E\int^{b}_{a}~_{V^{*}}\langle X(a),\varphi(t)v\rangle_{V}\d t+\lim_{n\rightarrow\infty}E\int^{b}_{a}\left(
     ~_{V_{1}^{*}}\langle A_{1}(\sigma,\overline{X}(\sigma,-n,0)),v\rangle_{V_{1}}\int^{b}_{\sigma}
     \varphi(t)\d t\right)\d\sigma\\
   & \quad +\lim_{n\rightarrow\infty}E\int^{b}_{a}\left(~_{V_{2}^{*}}\langle A_{2}(\sigma,\overline{X}(\sigma,-n,0)),
     v\rangle_{V_{2}}\int^{b}_{\sigma}\varphi(t)\d t\right)\d\sigma\\
   & \quad +E\left(\int^{b}_{a}\langle\int^{t}_{a}Z(\sigma)\d W(\sigma),\varphi(t)v\rangle_{H}\d t\right)\\
   & =E\int^{b}_{a}~_{V^{*}}\langle X(a),\varphi(t)v\rangle_{V}\d t
     +E\int^{b}_{a}\left(~_{V_{1}^{*}}\langle Y_{1}(\sigma),v\rangle_{V_{1}}\int^{b}_{\sigma}\varphi(t)\d t\right)\d\sigma\\
   & \quad +E\int^{b}_{a}\left(~_{V_{2}^{*}}\langle Y_{2}(\sigma),v\rangle_{V_{2}}\int^{b}_{\sigma}\varphi(t)
     \d t\right)\d\sigma+E\left(\int^{b}_{a}\langle \int^{t}_{a}Z(\sigma)\d W(\sigma),\varphi(t)v\rangle_{H}\d t\right).
\end{align*}
Let $Y(\sigma):=Y_{1}(\sigma)+Y_{2}(\sigma)\in W\subset V^{*}$, we have
$$X(t)=X(a)+\int^{t}_{a}Y(\sigma)\d\sigma+\int^{t}_{a}Z(\sigma)\d W(\sigma),
\quad \d t\otimes P\rm{\mbox{-}a.e.}$$
Thus, it remains to verify that
$$Y=A(\cdot,\overline{X}),\quad Z=B(\cdot,\overline{X}),\quad\d t\otimes P\rm{\mbox{-}a.e.}$$
To this end, for any $\phi\in K_{1}\cap K_{2}\cap L^{2}([a,b]\times\Omega,\d t\otimes P;H)$, we have
\begin{align*}
   & E\|X(t,-n,0)\|^{2}_{H}-E\|X(a,-n,0)\|^{2}_{H} \\
   & =E\int^{t}_{a}\Big(2_{V_{1}^{*}}\langle A_{1}(\sigma,\overline{X}(\sigma,-n,0)),
     \overline{X}(\sigma,-n,0)\rangle_{V_{1}}
     +2_{V_{2}^{*}}\langle A_{2}(\sigma,\overline{X}(\sigma,-n,0)),
     \overline{X}(\sigma,-n,0)\rangle_{V_{2}}\\
   & \qquad +\|B(\sigma,\overline{X}(\sigma,-n,0))\|^{2}_{L_{2}(U,H)}\Big)\d\sigma\\
   & \leq E\int^{t}_{a}\Big[2_{V_{1}^{*}}\langle A_{1}(\sigma,\overline{X}(\sigma,-n,0))-A_{1}(\sigma,\phi(\sigma)),
     \overline{X}(\sigma,-n,0)-\phi(\sigma)\rangle_{V_{1}}\\
   & \qquad+2_{V_{2}^{*}}\langle A_{2}(\sigma,\overline{X}(\sigma,-n,0))-A_{2}(\sigma,\phi(\sigma)),
     \overline{X}(\sigma,-n,0)-\phi(\sigma)\rangle_{V_{2}}\\
   & \qquad+\|B(\sigma,\overline{X}(\sigma,-n,0))-B(\sigma,\phi(\sigma))\|^{2}_{L_{2}(U,H)}\\
   & \qquad+2_{V_{1}^{*}}\langle A_{1}(\sigma,\overline{X}(\sigma,-n,0))-A_{1}(\sigma,\phi(\sigma)),
     \phi(\sigma)\rangle_{V_{1}}+2_{V_{1}^{*}}\langle A_{1}(\sigma,\phi(\sigma)),\overline{X}(\sigma,-n,0)\rangle_{V_{1}}\\
   & \qquad+2_{V_{2}^{*}}\langle A_{2}(\sigma,\overline{X}(\sigma,-n,0))-A_{2}(\sigma,\phi(\sigma)),\phi(\sigma)\rangle_{V_{2}}
     +2_{V_{2}^{*}}\langle A_{2}(\sigma,\phi(\sigma)),\overline{X}(\sigma,-n,0)\rangle_{V_{2}}\\
   & \qquad+2\langle B(\sigma,\overline{X}(\sigma,-n,0)),B(\sigma,\phi(\sigma))\rangle_{L_{2}(U,H)}
     -\|B(\sigma,\phi(\sigma))\|_{L_{2}(U,H)}^{2}\Big]\d\sigma.
\end{align*}

For every nonnegative $\psi\in L^{\infty}([a,b]\times\Omega,\d t\otimes P;\mathbb{R})$,
it follows from (H2$'$) that
\begin{align}\label{pin}
   & E\int^{b}_{a}\psi(t)\left(\|X(t,-n,0)\|_{H}^{2}-\|X(a,-n,0)\|^{2}_{H}\right)\d t  \\\nonumber
   & \leq E\Bigg(\int^{b}_{a}\psi(t)\int^{t}_{a}\Big(2_{V_{1}^{*}}\langle A_{1}(\sigma,\overline{X}(\sigma,-n,0))
     -A_{1}(\sigma,\phi(\sigma)),\phi(\sigma)\rangle_{V_{1}}\\\nonumber
   & \qquad +2_{V_{1}^{*}}\langle A_{1}(\sigma,\phi(\sigma)),\overline{X}(\sigma,-n,0)\rangle_{V_{1}}
     +2_{V_{2}^{*}}\langle A_{2}(\sigma,\overline{X}(\sigma,-n,0))-A_{2}(\sigma,\phi(\sigma)),
     \phi(\sigma)\rangle_{V_{2}}\\\nonumber
   & \qquad +2_{V_{2}^{*}}\langle A_{2}(\sigma,\phi(\sigma)),\overline{X}(\sigma,-n,0)\rangle_{V_{2}}
     +2\langle B(\sigma,\overline{X}(\sigma,-n,0)),B(\sigma,\phi(\sigma))\rangle_{L_{2}(U,H)}\\\nonumber
   & \qquad -\|B(\sigma,\phi(\sigma))\|^{2}_{L_{2}(U,H)}\Big)\d\sigma\d t\Bigg).
\end{align}
Using (1) we obtain
\begin{align*}
   & E\int_{a}^{b}\psi(t)\|X(t)\|_{H}^{2} \d t\\
   & =\lim\limits_{n\rightarrow\infty}E\int_{a}^{b}\langle\psi(t)X(t),X(t,-n,0)\rangle_{H}\d t\\
   & \leq \left(E\int_{a}^{b}\psi(t)\|X(t)\|_{H}^{2}\d t\right)^{\frac{1}{2}}
     \liminf_{n\rightarrow\infty}\left(E\int_{a}^{b}\psi(t)\|X(t,-n,0)\|_{H}^{2}\d t\right)^{\frac{1}{2}}.
\end{align*}
Then letting $n\rightarrow\infty$ in \eqref{pin}, we have
\begin{align}\label{Bountheq2}
   & E\int^{b}_{a}\psi(t)\left(\|X(t)\|^{2}_{H}-\|X(a)\|^{2}_{H}\right)\d t \\\nonumber
   & \leq E\Bigg(\int^{b}_{a}\psi(t)\int^{t}_{a}\Big(2_{V_{1}^{*}}\langle Y_{1}(\sigma)
     -A_{1}(\sigma,\phi(\sigma)),\phi(\sigma)\rangle_{V_{1}}+2_{V_{1}^{*}}\langle A_{1}(\sigma,\phi(\sigma)),
     \overline{X}(\sigma)\rangle_{V_{1}}\\\nonumber
   & \qquad +2_{V_{2}^{*}}\langle Y_{2}(\sigma)-A_{2}(\sigma,\phi(\sigma)),\phi(\sigma)\rangle_{V_{2}}
     +2_{V_{2}^{*}}\langle A_{2}(\sigma,\phi(\sigma)),\overline{X}(\sigma)\rangle_{V_{2}}\\\nonumber
   & \qquad +2\langle Z(\sigma),B(\sigma,\phi(\sigma))\rangle_{L_{2}(U,H)}
     -\|B(\sigma,\phi(\sigma))\|^{2}_{L_{2}(U,H)}\Big)\d\sigma\d t\Bigg).
\end{align}
And in view of the product rule, we get
\begin{align}\label{Bountheq3}
   & E\int^{b}_{a}\psi(t)\left(\|X(t)\|^{2}_{H}-\|X(a)\|^{2}_{H}\right)\d t \\\nonumber
   & =E\Bigg(\int^{b}_{a}\psi(t)\int^{t}_{a}\Big(2_{V_{1}^{*}}\langle Y_{1}(\sigma),\overline{X}(\sigma)\rangle_{V_{1}}
     +2_{V_{2}^{*}}\langle Y_{2}(\sigma),\overline{X}(\sigma)\rangle_{V_{2}}
     +\|Z(\sigma)\|^{2}_{L_{2}(U,H)}\Big)\d\sigma \d t\Bigg).
\end{align}
Therefore, \eqref{Bountheq2} and \eqref{Bountheq3} imply
\begin{equation}\label{ineq3}
  E\Bigg(\int^{b}_{a}\psi(t)\int^{t}_{a}\Big(2_{V^{*}}\langle Y(\sigma)-A(\sigma,\phi(\sigma)),
     \overline{X}(\sigma)-\phi(\sigma)\rangle_{V}
   +\|B(\sigma,\phi(\sigma))-Z(\sigma)\|_{L_{2}(U,H)}^{2}\Big)\d\sigma \d t\Bigg)\leq0.
\end{equation}
Taking $\phi=\overline{X}$ in \eqref{ineq3},  we have
$Z=B(\cdot,\overline{X})$, $\d t\otimes P$-a.e.
Then, applying \eqref{ineq3} to $\phi=\overline{X}-\epsilon\widetilde{\phi}v$ for $\epsilon>0$ and
$\widetilde{\phi}\in L^{\infty}([a,b]\times\Omega,\d t\otimes P;\mathbb{R})$, $v\in V$,
we have
\begin{equation*}
  E\Bigg(\int^{b}_{a}\psi(t)\int^{t}_{a}
   2_{V^{*}}\langle Y(\sigma)-A(\sigma,\overline{X}(\sigma)-\epsilon\widetilde{\phi}(\sigma)v),
   \epsilon\widetilde{\phi}(\sigma)v\rangle_{V}\d\sigma\d t\Bigg)\leq 0.
\end{equation*}
Dividing both sides by $\epsilon$ and letting $\epsilon\rightarrow0$,
according to Lebesgue's dominated convergence theorem, (H1) and (H4), we obtain
\begin{equation*}
  E\left(\int^{b}_{a}\psi(t)\int^{t}_{a}\widetilde{\phi}(\sigma) _{V^{*}}\langle Y(\sigma)
  -A(\sigma,\overline{X}(\sigma)),v\rangle_{V}\d\sigma\d t\right) \leq 0.
\end{equation*}
By the arbitrariness of $\psi$, $\widetilde{\phi}$ and $v$, we conclude that
$Y=A(\cdot,\overline{X})$, $\d t\otimes P$-a.e. This completes the existence proof, i.e.
$$X(t)=X(a)+\int^{t}_{a}A(\sigma,\overline{X}(\sigma))\d\sigma
+\int^{t}_{a}B(\sigma,\overline{X}(\sigma))\d W(\sigma),\quad \d t\otimes P\rm{\mbox{-}a.e.}$$
In view of the arbitrariness of interval $[a,b]\subset \mathbb{R}$, we conclude that $X(\cdot)$ is a
solution on $\mathbb{R}$. It follows from \eqref{ineq2} and \eqref{L2lim1} that $\sup\limits_{t\in\mathbb{R}}E\|X(t)\|_{H}^{2}<\infty$.

Now, we prove the uniqueness of $L^{2}$-bounded solution.
Suppose that $X(\cdot)$ and $Y(\cdot)$ are two $L^{2}$-bounded continuous solutions to
equation \eqref{eqSPDE1}, By (H2$'$) we have
\begin{align*}
   & E\|X(t)-Y(t)\|^{2}_{H} \\
   & =E\|X(t,-n,X(-n))-Y(t,-n,Y(-n))\|^{2}_{H} \\
   & \leq {\rm{e}}^{-\lambda(t+n)}E\|X(-n)-Y(-n)\|^{2}_{H}  \rightarrow 0, \quad {\rm{as}}~~~n\rightarrow\infty.
\end{align*}

The goal next is to prove that $\widehat{\mu}$ is unique with the properties (i) and (ii).
Note that
$$\sup_{t\in\mathbb{R}}\int_{H}\|x\|^{2}_{H}\widehat{\mu}(t)(\d x)=\sup_{t\in\mathbb{R}}E\|X(t)\|^{2}_{H}<\infty.$$
In view of the Chapman-Kolmogorov equation, we have
$$\mu(t,s,\L(X(s,-n,0)))=\L(X(t,-n,0)).$$
Then according to the Feller property (see, e.g. \cite[Proposition 4.2.10]{PR}), we get
$$\mu(t,s,\widehat{\mu}(s))=\widehat{\mu}(t).$$
Suppose that $\mu_{1}$ and $\mu_{2}$ satisfy properties (i) and (ii),
let $\zeta_{n,1}$ and $\zeta_{n,2}$ be random variables with the distributions $\mu_{1}(-n)$ and $\mu_{2}(-n)$ respectively.
Then consider the solutions $X(t,-n,\zeta_{n,1})$ and $X(t,-n,\zeta_{n,2})$ on $[-n,\infty)$,
we have
\begin{align*}
   & d_{BL}(\mu_{1}(t),\mu_{2}(t)) \\
   & =d_{BL}(\mu(t,-n,\mu_{1}(-n)),\mu(t,-n,\mu_{2}(-n))) \\
   & =\sup_{\|f\|_{BL}\leq1} \left|\int_{H}f(x)\d\left(\mu(t,-n,\mu_{1}(-n))-\mu(t,-n,\mu_{2}(-n))\right)\right| \\
   & =\sup_{\|f\|_{BL}\leq1} \left|\int_{\Omega}\left[f(X(t,-n,\zeta_{n,1}))-f(X(t,-n,\zeta_{n,2}))\right]\d P\right|\\
   & \leq \left(E\|X(t,-n,\zeta_{n,1})-X(t,-n,\zeta_{n,2})\|^{2}_{H}\right)^{1/2} \\
   & \leq {\rm{e}}^{-\frac{\lambda}{2}(t+n)}\left(E\|\zeta_{n,1}-\zeta_{n,2}\|^{2}_{H}\right)^{1/2}
     \rightarrow0, \quad {\rm{as}}~n\rightarrow\infty.
\end{align*}
Thus, $\mu_{1}(t)=\mu_{2}(t)$ for all $t\in\R$.
\end{proof}

\section{Recurrent solutions}

In this section, we show that the $L^{2}$-bounded solution
for equation \eqref{eqSPDE1} has the same character of recurrence as coefficients $A$ and $B$.

\subsection{Periodic solutions}

The following theorem shows that the $L^{2}$-bounded solution for equation \eqref{eqSPDE1}
is periodic in distribution provided the coefficients $A$ and $B$  are periodic.

\begin{theorem}\label{Pth}
Consider equation \eqref{eqSPDE1}.
Suppose that {\rm{(H1)}}, {\rm{(H2$'$)}}, {\rm{(H3)}},  {\rm{(H4)}} and {\rm{(HL)}} hold.
Assume further that the mappings $A$ and $B$ are $T$-periodic in $t$.
Then the unique $L^{2}$-bounded solution is $T$-periodic in distribution.

In particular, the unique $L^{2}$-bounded solution is stationary, provided $A$ and $B$ are independent of t.
\end{theorem}

\begin{proof}
Define the transition probability function as follows
$$p_{t,s}(x,\d y):=P\circ(X(t,s,x))^{-1}(\d y),\quad s\leq t,~ x\in H.$$
We now check that $p_{t,s}(x,\d y)$ is $T$-periodic,
provided the coefficients of equation \eqref{eqSPDE1} are $T$-periodic. For any $t\geq s$
\begin{align*}
  & X(t+T,s+T,x)\\
  & =x+\int_{s+T}^{t+T}A(\sigma,\overline{X}(\sigma,s+T,x))\d\sigma
    +\int_{s+T}^{t+T}B(\sigma,\overline{X}(\sigma,s+T,x))\d W(\sigma) \\
  & =x+\int_{s}^{t}A(\sigma+T,\overline{X}(\sigma+T,s+T,x))\d\sigma
    +\int_{s}^{t}B(\sigma+T,\overline{X}(\sigma+T,s+T,x))\d W(\sigma+T) \\
  & =x+\int_{s}^{t}A(\sigma,\overline{X}(\sigma+T,s+T,x))\d\sigma
    +\int_{s}^{t}B(\sigma,\overline{X}(\sigma+T,s+T,x))\d\widetilde{ W}(\sigma),
\end{align*}
where $\widetilde{ W}(\cdot)=W(\cdot+T)-W(T)$.

In order to indicate the dependence of the solution $X(t,s,x)$, $t\in[s,\infty)$ of equation
\eqref{eqSPDE1} on the Wiener process,
we write $X^{W}(t,s,x)$ instead of $X(t,s,x)$. Similarly, we write $p_{t,s}^{W}(x,\d y)$
instead of $p_{t,s}(x,\d y)$. So by the uniqueness of the solutions to equation \eqref{eqSPDE1}, for any
$t\in[s,\infty)$, $X^{W}(t+T,s+T,x)=X^{\widetilde{ W}}(t,s,x)$, $P$-a.e.
In particular, we have
\begin{align*}
  p_{t+T,s+T}^{W}(x,\d y) & =P\circ(X^{W}(t+T,s+T,x))^{-1}(\d y) \\
   & =P\circ(X^{\widetilde{W}}(t,s,x))^{-1}(\d y)\\
   & =P\circ(X^{W}(t,s,x))^{-1}(\d y)=p_{t,s}^{W}(x,\d y),
\end{align*}
where the third equality follows from Yamada-Watanabe theorem (see \cite{RSZ}).

Now we prove that $\widehat{\mu}(t)$, $t\in\R$ is $T$-periodic in $Pr(H)$, recalling that $\widehat{\mu}(t)=\L(X(t))$, $t\in\R$ is the distribution
of the unique $L^{2}$-bounded solution $X(\cdot)$ on $H$.
For any $\phi\in C_{b}(H)$, we have
\begin{align*}
  \int_{H}\phi(x)\widehat{\mu}(t+T)(\d x) & = \int_{H}\phi(x)(P\circ(X(t+T))^{-1})(\d x)\\
   & =\lim\limits_{n\rightarrow\infty}\int_{H}\phi(x)(P\circ(X(t+T,-n+T,0))^{-1})(\d x)\\
   & =\lim\limits_{n\rightarrow\infty}\int_{H}\phi(x)(P\circ(X(t,-n,0))^{-1})(\d x)\\
   & =\int_{H}\phi(x)\widehat{\mu}(t)(\d x).
\end{align*}
The $T$-periodicity of the distribution of $X(\cdot)$ on $C(\R,H)$ now follows from the uniqueness in law of the solutions
for equation \eqref{eqSPDE1}.
The proof is complete.
\end{proof}

\subsection{Almost periodic solutions}

In the sequel, we show that the $L^{2}$-bounded solution of equation \eqref{eqSPDE1}
is almost periodic in distribution, if the coefficients $A$ and $B$ are uniformly almost periodic.
To this end, we need the tightness of the family of distributions
$\{P\circ[X(t)]^{-1}\}_{t\in\mathbb{R}}$.
Note that $\{P\circ[X(t)]^{-1}\}_{t\in\mathbb{R}}$ is tight provided ${\rm{dim}}H<\infty$.
But when ${\rm{dim}}H=\infty$, we need the following condition (H5) to get the tightness of
$\{P\circ[X(t)]^{-1}\}_{t\in\mathbb{R}}$. This condition was used in \cite{LiuW101} to study
the invariance of subspaces and in \cite{GLR} to study random attractors.

(H5) Assume that there exists a closed subset $S\subset H$ equipped with the norm $\|\cdot\|_{S}$
such that $V\subset S$ is continuous and $S\subset H$ is compact. Let $T_{n}$ be a sequence of
positive definite self-adjoint operators on $H$ such that for each $n\geq1$,
$$\langle x,y\rangle_{n}:=\langle x,T_{n}y\rangle_{H}, \quad x,y\in H,$$
defines a new inner product on $H$. Assume further that the norms $\|\cdot\|_{n}$ generated by
$\langle ~,~\rangle_{n}$ are all equivalent to $\|\cdot\|_{H}$ and for all $x\in S$ we have
$$\|x\|_{n} \uparrow \|x\|_{S} \quad {\rm{as}}~n\rightarrow\infty.$$
Furthermore, we suppose that for each $n\geq1$, $T_{n}: V\rightarrow V$ is continuous and
there exist constants $C>0$, $M_{0}>0$ such that for all $u,v\in V$, $t\in\mathbb{R}$
\begin{equation*}
  2_{V^{*}}\langle A(t,v),T_{n}v\rangle_{V}+\|B(t,v)\|^{2}_{L_{2}(U,H_{n})}\leq-C\|v\|_{n}^{2}+M_{0}.
\end{equation*}

Let $H_{n}:=(H,\langle~,~\rangle_{n}).$
We denote by $i_{n}$ the Riesz isomorphism from $H_{n}$ into $H_{n}^{*}$.
Similarly, $i:H\rightarrow H^{*}$.

\begin{lemma}[See Liu \cite{LiuW101}]\label{newgf}
If $T_{n}: V\rightarrow V$ is continuous, then $i_{n}\circ i^{-1} : H^{*}\rightarrow H_{n}^{*}$
is continuous with respect to $\|\cdot\|_{V^{*}}$.
Therefore, there exists a unique extension $I_{n}$ of $ i_{n}\circ i^{-1}$ on $V^{*}$ such that
for all $f\in V^{*}$, $v\in V$
\begin{equation}\label{11}
  ~_{V^{*}}\langle I_{n}f,v\rangle_{V}=~_{V^{*}}\langle f,T_{n}v\rangle_{V}.
\end{equation}
\end{lemma}

\begin{prop}\label{tightprop}
Consider equation \eqref{eqSPDE1}.
Suppose that {\rm{(H1)}}, {\rm{(H2$'$)}}, {\rm{(H3)}},  {\rm{(H4)}} and {\rm{(H5)}} hold,
and that $\zeta_{s}\in L^{2}(\Omega,\mathcal{F}_{s},P;H)$. Let  $X(t,s,\zeta_{s})$, $t\geq s$
be the solution to equation \eqref{eqSPDE1} with initial condition $X(s)=\zeta_s$. Then the $L^{2}$-bounded solution $X(\cdot)$ satisfies
$$\sup_{t\in\mathbb{R}}E\|X(t)\|_{S}^{2}<\infty.$$

In particular, the family of distributions $\{{P\circ[X(t)]^{-1}}\}_{t\in\mathbb{R}}$ is tight.
\end{prop}

\begin{proof}
Note that $X(t,s,\zeta_{s})$, $t\geq s$ satisfies
\begin{equation}\label{newgfeq}
  iX(t,s,\zeta_{s})=i\zeta_{s}+\int^{t}_{s}A(\sigma,\overline{X}(\sigma,s,\zeta_{s}))\d\sigma
+i\left(\int_{s}^{t}B(\sigma,\overline{X}(\sigma,s,\zeta_{s}))\d W(\sigma)\right), \quad t\geq s.
\end{equation}
According to Lemma \ref{newgf}, applying $I_{n}$ to \eqref{newgfeq} we obtain
\[
i_{n}X(t,s,\zeta_{s})=i_{n}\zeta_{s}+\int^{t}_{s}I_{n}A(\sigma,\overline{X}(\sigma,s,\zeta_{s}))\d\sigma
+i_{n}\left(\int_{s}^{t}B(\sigma,\overline{X}(\sigma,s,\zeta_{s}))\d W(\sigma)\right).
\]
Then using It\^o's formula on the new Gelfand triple
\[
V\subset H_{n}\simeq H_{n}^{*} \subset V^{*},
\]
we get
\begin{align*}
    & \|X(t,s,\zeta_{s})\|_{n}^{2} \\
    & =\|\zeta_{s}\|_{n}^{2}+\int^{t}_{s}\Big(2_{V^{*}}\langle I_{n}A(\sigma,\overline{X}(\sigma,s,\zeta_{s})),
    \overline{X}(\sigma,s,\zeta_{s})\rangle_{V}
    +\|B(\sigma,\overline{X}(\sigma,s,\zeta_{s}))\|_{L_{2}(U,H_{n})}^{2}\Big) \d\sigma\\
    &\quad  +2\int^{t}_{s}\langle X(\sigma,s,\zeta_{s}),B(\sigma,\overline{X}(\sigma,s,\zeta_{s}))\d W(\sigma)\rangle_{n}.
  \end{align*}
See \cite{LiuW101} for details. By the product rule, \eqref{11} and (H5) we have
\begin{align*}
   & E\left({\rm{e}}^{C(t-s)}\|X(t,s,\zeta_{s})\|^{2}_{n}\right) \\
   & =E\|\zeta_{s}\|_{n}^{2}+E\int^{t}_{s}{\rm{e}}^{C(\sigma-s)}\Big(2_{V^{*}}\langle
     I_{n}A(\sigma,\overline{X}(\sigma,s,\zeta_{s})),\overline{X}(\sigma,s,\zeta_{s})\rangle_{V}\\
   & \qquad +\|B(\sigma,\overline{X}(\sigma,s,\zeta_{s}))\|^{2}_{L_{2}(U,H_{n})}\Big)\d\sigma
     +E\int_{s}^{t}C{\rm{e}}^{C(\sigma-s)}\|X(\sigma,s,\zeta_{s})\|^{2}_{n}\d\sigma\\
   & \leq E\|\zeta_{s}\|_{n}^{2}+\int_{s}^{t}{\rm{e}}^{C(\sigma-s)}M_{0}\d\sigma\\
   & \leq E\|\zeta_{s}\|_{n}^{2}+\frac{M_{0}}{C}{\rm{e}}^{C(t-s)}.
\end{align*}
Therefore, we obtain
$$E\|X(t,s,\zeta_{s})\|^{2}_{n}\leq {\rm{e}}^{-C(t-s)}E\|\zeta_{s}\|_{n}^{2}+\frac{M_{0}}{C}.$$
In particular, we have
$$E\|X(t,-m,0)\|^{2}_{n}\leq \frac{M_{0}}{C}.$$
Thus we may assume, going if necessary to a subsequence, that
\[
X(t,-m,0)\rightarrow X(t)\quad {\rm {weakly ~ in}} ~ L^{2}(\Omega,P;H_{n}).
\]
Then Fatou's lemma yields that
$$E\|X(t)\|^{2}_{n}\leq\liminf_{m\rightarrow\infty}E\|X(t,-m,0)\|^{2}_{n}\leq\frac{M_{0}}{C}.$$
Moreover, we have
\begin{equation*}
  E\|X(t)\|^{2}_{S}= E\lim_{n\rightarrow\infty}\|X(t)\|^{2}_{n}
  \leq\liminf_{n\rightarrow\infty}E\|X(t)\|^{2}_{n}\leq\frac{M_{0}}{C}.
\end{equation*}

The tightness of $\{{P\circ[X(t)]^{-1}}\}_{t\in\mathbb{R}}$ is an easy consequence of the
compactness of the embedding $S\subset H$.
\end{proof}

\begin{remark}\rm
Note that Gess \cite{Gess} gave a different technique to obtain the compactness of random dynamical
systems generated by stochastic singular evolution equations. That is, he got the
compactness of the solution mapping for any fixed sample point. Since the recurrence
we are concerned with in this paper is in distribution sense instead of in pathwise sense,
it seems that the technique in \cite{Gess} is not applicable to our problem.
\end{remark}

\begin{prop}\label{CBprop}
Consider equation \eqref{eqSPDE1}.
Suppose that $A$, $B$, $A_{n}$, $B_{n}$ satisfy {\rm{(H1)}}, {\rm{(H2$'$)}}, {\rm{(H3)}},  {\rm{(H4)}}
and {\rm{(HL)}} with the same constants $\lambda$, $c_{1}$, $c_{2}$, $c_{3}$, $c_{2}'$, $c_{3}'$, $M_{0}$,
$\alpha_{i}$, $i=1,2$ and $L_{B}$. Let $X(\cdot)$, $X_{n}(\cdot)$ be the $L^{2}$-bounded solutions
of equation \eqref{eqSPDE1} corresponding to $A$, $B$ and $A_{n}$, $B_{n}$ respectively.
Assume in addition that
\begin{enumerate}
  \item $\lim\limits_{n\rightarrow\infty}\|A_{i,n}(t,x)-A_{i}(t,x)\|_{V_{i}^{*}}=0$
        for all $x\in V$, $t\in \mathbb{R}$, $i=1,2$;
  \item $\lim\limits_{n\rightarrow\infty}\|B_{n}(t,x)-B(t,x)\|_{L_{2}(U,H)}=0$
        for all $x\in V$, $t\in \mathbb{R}$;
  \item for each $t\in \mathbb{R}$ the family of distributions
        $\{P\circ[X_{n}(t)]^{-1}\}_{n\in\mathbb{N}}$ is tight.
\end{enumerate}
Then
$$\lim_{n\rightarrow\infty}d_{BL}(\mathcal{L}(X_{n}),\mathcal{L}(X))=0
\quad {\rm{in}} ~ Pr(C(\mathbb{R},H)).$$

In particular,
$$\lim_{n\rightarrow\infty}d_{BL}(\mathcal{L}(X_{n}(t+\cdot)),\mathcal{L}(X(t+\cdot)))=0
\quad {\rm{in}}~ Pr(C(\mathbb{R},H))~{\rm{for~all}}~t\in\mathbb{R}.$$
\end{prop}

\begin{proof}
According to Remark \ref{Remc}, we only need to prove that
$\lim\limits_{n\rightarrow\infty}d_{BL}(\mathcal{L}(X_{n}(t)),\mathcal{L}(X(t)))=0$
in $Pr(H)$ for every $t\in\R$. To this end, it suffices to show that for every sequence
$\gamma'=\{\gamma_{k}'\}:=\{\gamma_{k}'\}_{k=1}^{\infty}\subset\N$,
there exists a subsequence $\gamma=\{\gamma_{k}\}$ of $\gamma'$ such that for every $t\in\R$
$\lim\limits_{k\rightarrow\infty}d_{BL}(\mathcal{L}(X_{\gamma_{k}}(t)),\mathcal{L}(X(t)))=0$ in $Pr(H)$.

For every $r\geq1$, according to the tightness of  $\{\L(X_{\gamma_{k}'}(-r))\}$, there exists
a subsequence $\{\gamma_{k}\}\subset\gamma'$ such that
$\L(X_{\gamma_{k}}(-r))$ converges weakly to some  probability measure $\mu_{r}$ in $Pr(H)$.
Let $\xi_{r}$ be a random variable with distribution $\mu_{r}$.
Define $Y_{r}(t):=X(t,-r,\xi_{r})$, recalling that
$X(t,-r,\xi_{r})$, $t\in[-r,+\infty)$ is the solution to the following Cauchy problem
\begin{equation*}
  \left\{
   \begin{aligned}
   &\ \d X(t)=A(t,X(t))\d t+B(t,X(t))\d W(t)\\
  &\ X(-r)=\xi_{r}.
   \end{aligned}
   \right.
  \end{equation*}
In view of Theorem \ref{conlemma}, we have
$$\lim\limits_{k\rightarrow\infty}d_{BL}(\L(X_{\gamma_{k}}),\L(Y_{r}))=0
\quad{\rm {in}}~ Pr(C([-r,+\infty),H)).$$
Since $\{\L(X_{\gamma_{k}}(-r-1))\}$ is tight, going if necessary to a subsequence,
we can assume that $\L(X_{\gamma_{k}}(-r-1))$ converges weakly to some  probability
measure $\mu_{r+1}$ in $Pr(H)$.
Let $\xi_{r+1}$ be a random variable with distribution
$\mu_{r+1}$. In light of Theorem \ref{conlemma}, we have
 $$\lim\limits_{k\rightarrow\infty}d_{BL}(\L(X_{\gamma_{k}}),\L(Y_{r+1}))=0
 \quad{\rm {in}} ~ Pr(C([-r-1,+\infty),H)),$$
where $Y_{r+1}(t):=X(t,-r-1,\xi_{r+1})$, $t\in[-r-1,+\infty)$.
Therefore, we have $d_{BL}(\L(Y_{r}),\L(Y_{r+1}))=0$ in $Pr(C([-r,+\infty),H))$. In particular,
$\L(Y_{r}(t))=\L(Y_{r+1}(t))$, for all $t\geq-r$.

Define $\nu(t):=\L(Y_{r}(t))$,  $t\geq-r$. We use a standard diagonal argument to
extract a subsequence which we still denote by $\{X_{\gamma_{k}}\}$ satisfying
 $$\lim\limits_{k\rightarrow\infty}d_{BL}(\L(X_{\gamma_{k}}(t)),\nu(t))=0 \quad{\rm in} ~Pr(H)$$
for every $t\in\R$.
Note that $\sup\limits_{t\in\R}\int_{H}\|x\|_{H}^{2}\nu(t)(\d x)<+\infty$. And we have $P$-a.e.
\begin{equation*}
  Y_{r}(t)=Y_{r}(s)+\int_{s}^{t}A(\sigma,\overline{Y}_{r}(\sigma))\d\sigma
  +\int_{s}^{t}B(\sigma,\overline{Y}_{r}(\sigma))\d W(\sigma), \quad {\rm{where}}~ t\geq s\geq-r.
\end{equation*}
By the uniqueness in law of the solutions for equation \eqref{eqSPDE1}, we have
$\L(Y_{r}(t))=\mu(t,s,\L(Y_{r}(s))$, $t\geq s\geq-r$,
i.e. $\nu(t)=\mu(t,s,\nu(s))$, $t\geq s$. In view of Theorem \ref{Boundedth},
we obtain $\nu=\widehat{\mu}$. Therefore, for every $t\in\R$, we have
\begin{equation*}
  \lim\limits_{k\rightarrow\infty}d_{BL}(\L(X_{\gamma_{k}}(t)),\L(X(t)))=0 \quad {\rm{in}}~ Pr(H).
\end{equation*}
\end{proof}

The following result shows that the $L^{2}$-bounded solution is almost periodic in distribution
provided $A$ and $B$ are uniformly almost  periodic.

\begin{theorem}\label{APth}
Consider equation \eqref{eqSPDE1}.
Suppose that {\rm{(H1)}}, {\rm{(H2$'$)}}, {\rm{(H3)}},  {\rm{(H4)}}, {\rm{(H5)}} and {\rm{(HL)}} hold.
Assume further that the mappings $A_{1}$, $A_{2}$ and $B$ are uniformly almost periodic.
Then the unique $L^{2}$-bounded solution is almost periodic in distribution.
\end{theorem}

\begin{proof}
Let $\gamma'=\{\gamma_{n}'\}$ and $\beta'=\{\beta_{n}'\}$ be two sequences in $\mathbb{R}$.
By Definition \ref{apaalaw} and Theorem \ref{BochBohr}, it suffices to show that there exist two subsequences
$\gamma=\{\gamma_{n}\}\subset\gamma'=\{\gamma_{n}'\}$ and $\beta=\{\beta_{n}\}\subset\beta'=\{\beta_{n}'\}$
with the same indexes such that for every  $t\in \mathbb{R}$
\begin{equation*}
  \lim_{n\rightarrow\infty}\lim_{m\rightarrow\infty}\widehat{\mu}(t+\gamma_{n}+\beta_{m}),
  \quad \lim_{n\rightarrow\infty}\widehat{\mu}(t+\gamma_{n}+\beta_{n})
\end{equation*}
exist and equal.

In fact, since $A_{1}$, $A_{2}$ and $B$ are uniformly almost periodic,
there exist $\gamma=\{\gamma_{n}\}\subset\gamma'$ and $\beta=\{\beta_{n}\}\subset\beta'$
with the same indexes such that $T_{\beta}A_{i}$, $T_{\beta}B$, $T_{\gamma}T_{\beta}A_{i}$,
$T_{\gamma}T_{\beta}B$, $T_{\gamma+\beta}A_{i}$, $i=1,2$ and
$T_{\gamma+\beta}B$ exist uniformly with respect to $t\in\mathbb{R}$ and
$x\in Q$, where $Q$ is an arbitrary compact subset of $V$. Furthermore, we have
$$T_{\gamma}T_{\beta}A_{i}=T_{\gamma+\beta}A_{i},~i=1,2,\quad T_{\gamma}T_{\beta}B=T_{\gamma+\beta}B.$$
It can be verified that
$T_{\beta}A_{i}$, $T_{\beta}B$, $T_{\gamma}T_{\beta}A_{i}$, $i=1,2$ and $T_{\gamma}T_{\beta}B$
satisfy (H1), (H2$'$), (H3), (H4), (H5) and (HL) with the same constants $\lambda$, $c_{1}$, $c_{2}$,
$c_{3}$, $c'_{2}$, $c'_{3}$, $M_{0}$, $\alpha_{i}$, $i=1,2$ and $L_{B}$.

Let $Y(\cdot)$, $Z_{1}(\cdot)$ be the unique $L^{2}$-bounded solutions of equation \eqref{eqSPDE1}
with coefficients $T_{\beta}A_{i}$, $T_{\beta}B$ and $T_{\gamma}T_{\beta}A_{i}$,
$T_{\gamma}T_{\beta}B$, $i=1,2$, respectively. In view of Proposition \ref{CBprop} and the uniqueness
in law of the solutions for equation \eqref{eqSPDE1}, we obtain
\begin{equation}\label{APeq1}
  \lim\limits_{m\rightarrow\infty}d_{BL}(\L(X(\cdot+\beta_{m})),\L(Y))=0
  \quad {\rm{in}}~ Pr(C(\R,H))
\end{equation}
and
\begin{equation}\label{APeq2}
  \lim\limits_{n\rightarrow\infty}d_{BL}(\L(Y(\cdot+\gamma_{n})),\L(Z_{1}))=0
  \quad {\rm{in}}~ Pr(C(\R,H)).
\end{equation}
Similarly, we have
\begin{equation}\label{APeq3}
  \lim\limits_{n\rightarrow\infty}d_{BL}(\L(X(\cdot+\gamma_{n}+\beta_{n})),\L(Z_{2}))=0
  \quad {\rm{in}}~ Pr(C(\R,H)),
\end{equation}
where $Z_{2}(\cdot)$ is the unique $L^{2}$-bounded solution to the following equation
\begin{equation*}
  \d X(t)=T_{\gamma+\beta}A(t,X(t))\d t+T_{\gamma+\beta}B(t,X(t))\d W(t).
\end{equation*}
Since the $L^{2}$-bounded solution to equation \eqref{eqSPDE1} is unique, \eqref{APeq1}--\eqref{APeq3} imply
\begin{align*}
  & \lim\limits_{n\rightarrow\infty}\lim\limits_{m\rightarrow\infty}
    d_{BL}(\L(X(\cdot+\gamma_{n}+\beta_{m})),\L(Z))\\
  & =\lim\limits_{n\rightarrow\infty}d_{BL}(\L(X(\cdot+\gamma_{n}+\beta_{n})),\L(Z))=0
    \quad {\rm{in}}~ Pr(C(\R,H)),
\end{align*}
where $Z:=Z_{1}=Z_{2}$.
Thus
 $$T_{\gamma}T_{\beta}\widehat{\mu}=T_{\gamma+\beta}\widehat{\mu}.$$
The proof is complete.
\end{proof}

\subsection{Almost automorphic solutions}

In this subsection, suppose that mappings $A$ and $B$ of
equation \eqref{eqSPDE1} are uniformly almost automorphic.
Then we prove that the $L^{2}$-bounded solution $X(\cdot)$ is almost automorphic in distribution.

\begin{theorem}\label{AAth}
Consider equation \eqref{eqSPDE1}.
Assume that {\rm{(H1)}}, {\rm{(H2$'$)}}, {\rm{(H3)}},  {\rm{(H4)}}, {\rm{(H5)}} and {\rm{(HL)}} hold.
Suppose further that mappings $A_{i}$, $i=1,2$ and $B$ are uniformly almost automorphic.
Then the unique $L^{2}$-bounded solution is almost automorphic in distribution.
\end{theorem}

\begin{proof}

Since $A_{1}$, $A_{2}$ and $B$ are uniformly almost automorphic, for any sequence $\beta'=\{\beta'_{n}\}$ in $\R$
there exists a subsequence $\beta=\{\beta_{n}\}\subset\beta'$ such that
\begin{equation*}
  \lim\limits_{n\rightarrow\infty}A_{i}(t+\beta_{n},x)=\widetilde{A}_{i}(t,x),~i=1,2,\quad
  \lim\limits_{n\rightarrow\infty}B(t+\beta_{n},x)=\widetilde{B}(t,x)
\end{equation*}
and
\begin{equation*}
  \lim\limits_{n\rightarrow\infty}\widetilde{A}_{i}(t-\beta_{n},x)=A_{i}(t,x),~i=1,2,\quad
  \lim\limits_{n\rightarrow\infty}\widetilde{B}(t-\beta_{n},x)=B(t,x).
\end{equation*}
These limits exist uniformly with respect to $[a,b]\times Q$, where $[a,b]$ is an arbitrary finite
interval and $Q$ an arbitrary compact subset of $V$.
Note that $\widetilde{A}_{i}$, $i=1,2$ and $\widetilde{B}$ satisfy (H1), (H2$'$), (H3),
(H4), (H5) and (HL) with the same constants $\lambda$, $c_{1}$, $c_{2}$, $c_{3}$, $c'_{2}$,
$c'_{3}$, $M_{0}$, $\alpha_{i},~i=1,2$ and $L_{B}$.

Let $Y$ be the unique $L^{2}$-bounded solution of  equation \eqref{eqSPDE1} corresponding
to $\widetilde{A}_{i}$, $i=1,2$ and $\widetilde{B}$.
In view of Proposition \ref{CBprop} and the uniqueness in law of the solutions for equation \eqref{eqSPDE1},
we obtain
\begin{equation*}
  \lim\limits_{n\rightarrow\infty}d_{BL}(\L(X(\cdot+\beta_{n})),\L(Y))=0 \quad {\rm{in}}~Pr(C(\R,H)).
\end{equation*}
Similarly, we have
\begin{equation*}
  \lim\limits_{n\rightarrow\infty}d_{BL}(\L(Y(\cdot-\beta_{n})),\L(X))=0 \quad {\rm{in}} ~ Pr(C(\R,H)).
\end{equation*}
The proof is complete.
\end{proof}

\section{SPDEs with additive noise}

In this section we consider the following stochastic partial differential equation
driven by additive noise
\begin{equation}\label{spdean}
  \d X(t)=A(t, X(t))\d t+B(t)\d W(t).
\end{equation}
Assume that $V=V_{1}=V_{2}$. Then we have $\alpha_{1}=\alpha_{2}=:\alpha$,
$c_{2}=c_{2}'=:c_{2}$, $c_{3}=c'_{3}=:c_{3}$.
For equation (\ref{spdean}), we can weaken the strictly monotone
condition (H2$'$) to the following strong monotone condition.
This condition has a much wider application; see e.g. \cite{GLR}.

(H2$''$) (Strong monotonicity) There exist constants $r\geq2$ and $\lambda>0$
such that for all $u$, $v\in V$, $t\in\mathbb{R}$
\begin{equation*}
  2_{V^{*}}\langle A(t,u)-A(t,v),u-v\rangle_{V}\leq-\lambda\|u-v\|_{H}^{r}.
\end{equation*}

Note that when $r=2$, (H2$''$) is the same as (H2$'$). Therefore, we only consider
the case $r>2$ in this section.

\begin{lemma}\label{anesti}
Consider equation \eqref{spdean}. Assume that {\rm{(H1)}, (H2$''$), (H3)} and {\rm(H4)} hold,
and there exists a constant $M_{3}>0$ such that
\begin{equation*}
  \|B(t)\|_{L_{2}(U,H)}^{2}\leq M_{3}.
\end{equation*}
Let $\zeta_{s}\in L^{2}(\Omega,\mathcal{F}_{s},P;H)$. Suppose that
$X(t,s,\zeta_{s})$, $t\geq s$ is a solution to equation \eqref{spdean} with initial condition $X(s)=\zeta_{s}$.
Then for any $\eta>0$, there exists a constant $M_{4}\geq0$ depending only on
$r$, $\eta$, $\alpha$, $c_{1}$, $c_{2}$, $M_{0}$ and $M_{3}$, such that
\begin{equation}\label{ANb}
  E\|X(t,s,\zeta_{s})\|_{H}^{2}\leq E\|\zeta_{s}\|_{H}^{2}{\rm{e}}^{-\eta(t-s)}+M_{4}.
\end{equation}
\end{lemma}

\begin{proof}
By (H3), (H4), (H2$''$) and Young's inequality, we have
\begin{align*}
   & 2_{V^{*}}\langle A(t,u),u\rangle_{V}+\|B(t)\|^{2}_{L_{2}(U,H)} \\
   & \leq2_{V^{*}}\langle A(t,u)-A(t,0),u\rangle_{V}+2_{V^{*}}\langle A(t,0),u\rangle_{V}+M_{3}\\
   & \leq-\lambda\|u\|_{H}^{r}+2\|A(t,0)\|_{V^{*}}\|u\|_{V}+M_{3}\\
   & \leq-\lambda\|u\|_{H}^{r}+\frac{2(\alpha-1)}{\alpha}\|A(t,0)\|_{V^{*}}^{\frac{\alpha}{\alpha-1}}
     +\frac{2}{\alpha}\|u\|_{V}^{\alpha}+M_{3}\\
   & \leq-\lambda\|u\|_{H}^{r}+\frac{2(\alpha-1)}{\alpha}M_{0}^{\frac{\alpha}{\alpha-1}}
     +\frac{2c_{1}}{\alpha c_{2}}\|u\|_{H}^{2}+\frac{2M_{0}}{\alpha c_{2}}\\
   & \quad-\frac{2}{\alpha c_{2}}\left(2_{V^{*}}\langle A(t,u),u\rangle_{V}+\|B(t)\|^{2}_{L_{2}(U,H)}\right)+M_{3}.
\end{align*}
Therefore, there exists a constant $\widetilde{M}_{0}=\widetilde{M}_{0}(\alpha,c_{2},M_{0},M_{3})$ such that
\begin{equation}\label{aaa}
  2_{V^{*}}\langle A(t,u),u\rangle_{V}+\|B(t)\|^{2}_{L_{2}(U,H)}\leq-\frac{\alpha c_{2}}{\alpha c_{2}+2}
  \lambda\|u\|_{H}^{r}+\frac{2c_{1}}{\alpha c_{2}+2}\|u\|_{H}^{2}+\widetilde{M}_{0}.
\end{equation}

According to Young's inequality and \eqref{aaa}, we obtain
\begin{align*}
   & E\left({\rm{e}}^{\eta(t-s)}\|X(t,s,\zeta_{s})\|_{H}^{2}\right) \\
   & =E\|\zeta_{s}\|_{H}^{2}+\int_{s}^{t}\eta {\rm{e}}^{\eta(\sigma-s)}E\|X(\sigma,s,\zeta_{s})\|_{H}^{2}\d\sigma\\
   & \quad +E\int_{s}^{t}{\rm{e}}^{\eta(\sigma-s)}\left(2_{V^{*}}\langle A(\sigma,\overline{X}(\sigma,s,\zeta_{s})),
     \overline{X}(\sigma,s,\zeta_{s})\rangle_{V}+\|B(\sigma)\|^{2}_{L_{2}(U,H)}\right)\d\sigma\\
   & \leq E\|\zeta_{s}\|_{H}^{2}+\int_{s}^{t}{\rm{e}}^{\eta(\sigma-s)}\widetilde{M}_{0}\d\sigma\\
   & \quad +E\int_{s}^{t} {\rm{e}}^{\eta(\sigma-s)}\left(\left(\eta+\frac{2c_{1}}{\alpha c_{2}+2}\right)
     \|\overline{X}(\sigma,s,\zeta_{s})\|_{H}^{2}
     -\frac{\alpha c_{2}}{\alpha c_{2}+2}\lambda\|\overline{X}(\sigma,s,\zeta_{s})\|_{H}^{r}\right)\d\sigma\\
   & \leq E\|\zeta_{s}\|_{H}^{2}+\int_{s}^{t}{\rm{e}}^{\eta(\sigma-s)}\widetilde{M}_{0}\d\sigma
     +E\int_{s}^{t} {\rm{e}}^{\eta(\sigma-s)}\bigg[-\frac{\alpha c_{2}}{\alpha c_{2}+2}
     \lambda\|\overline{X}(\sigma,s,\zeta_{s})\|_{H}^{r}\\
   & \qquad +\epsilon\|\overline{X}(\sigma,s,\zeta_{s})\|_{H}^{r}
     +C_{\epsilon}\left(\eta+\frac{2c_{1}}{\alpha c_{2}+2}\right)^{\frac{r}{r-2}}\bigg]\d\sigma.\\
\end{align*}
Choosing $\epsilon<\frac{\alpha c_{2}}{\alpha c_{2}+2}\lambda$, we get
\begin{equation*}
  E\left({\rm{e}}^{\eta(t-s)}\|X(t,s,\zeta_{s})\|_{H}^{2}\right)\leq E\|\zeta_{s}\|_{H}^{2}
  +\int_{s}^{t}{\rm{e}}^{\eta(\sigma-s)}C_{1}\d\sigma,
\end{equation*}
where the constant $C_{1}$ depends only on $\epsilon$, $r$, $\eta$, $\alpha$, $c_{1}$, $c_{2}$
and $\widetilde{M}_{0}$.
It follows that
\begin{equation*}
 E\|X(t,s,\zeta_{s})\|_{H}^{2}\leq E\|\zeta_{s}\|_{H}^{2}{\rm{e}}^{-\eta(t-s)}+M_{4},
\end{equation*}
where $M_{4}=\frac{C_{1}}{\eta}$.
\end{proof}

\begin{lemma}\label{adcont}
Consider equation \eqref{spdean}.
Assume that {\rm{(H1)}}, {\rm{(H2$''$)}}, {\rm{(H3)}} and {\rm{(H4)}} hold.
Let $X$ and $Y$ be solutions of equation \eqref{spdean}.
Then for any $u\leq \sigma\leq t$ we have the estimate
\begin{align}\label{ANSCP}
   & E\|X(t,u,X(u))-Y(t,\sigma,Y(\sigma))\|^{2}_{H} \\\nonumber
   & \leq E\|X(\sigma,u,X(u))-Y(\sigma)\|_{H}^{2}\wedge \left\{\frac{\lambda}{2}(r-2)
   (t-\sigma)\right\}^{-\frac{2}{r-2}}.\nonumber
\end{align}

In particular, for any $t\in\R$ there exists some random variable $X(t)$ such that
\begin{equation}\label{L2lim2}
  X(t,-n,0)\rightarrow X(t)\quad {\rm{in}}~ L^{2}(\Omega,P;H)
  ~ {\rm{as}}~ n\rightarrow\infty.
\end{equation}

\end{lemma}

\begin{proof}
The proof is analogous to Lemma 2.5 in \cite{GLR}.
\end{proof}

Similar to Section 3, we will prove that the limit process $X(\cdot)$ in \eqref{L2lim2} is a solution
to equation \eqref{spdean}. To this end, we also need some uniform estimates.

\begin{lemma}\label{anbound}
Consider equation \eqref{spdean}.
Assume that the conditions of Lemma \ref{anesti} hold. For any fixed interval
$[a,b]\subset\mathbb{R}$ denote
$$K:=L^{\alpha}([a,b]\times\Omega,\d t\otimes P;V), \quad K^{*}:=
L^{\frac{\alpha}{\alpha-1}}([a,b]\times\Omega,\d t\otimes P;V^{*}).$$
Then there exists a constant $M_{2}$, depending only on $M_{0}$, $M_{4}$, $c_{1}$,
$c_{2}$  and $[a,b]$, such that
\begin{equation*}
  \sup_{t\in\R}E\|X(t,-n,0)\|_{H}^{2}+\|\overline{X}(\cdot,-n,0)\|_{K}
  +\|A(\cdot,\overline{X}(\cdot,-n,0))\|_{K^{*}}\leq M_{2}
\end{equation*}
for all $-n\leq a$.
\end{lemma}

\begin{proof}
In view of  \eqref{ANb}, we have
$$E\|X(t,-n,0)\|_{H}^{2}\leq M_{4}.$$
Similar to the proof of Lemma \ref{Bestimate}, we obtain
$$\|\overline{X}(\cdot,-n,0)\|_{K}+\|A(\cdot,\overline{X}(\cdot,-n,0))\|_{K^{*}}\leq M_{2}.$$
The proof is complete.
\end{proof}

Under the conditions of Lemma \ref{anesti}, we obtain the following conclusion
which is as same as Theorem \ref{Boundedth}.
\begin{theorem}
Suppose that the conditions of Lemma \ref{anesti} hold,
then there exists a unique $L^{2}$-bounded continuous $H$-valued solution $X(t)$, $t\in\mathbb{R}$ to
equation \eqref{spdean}. Moreover, the mapping $\widehat{\mu}: \mathbb{R}\rightarrow Pr(H)$
defined by $\widehat{\mu}(t):=P\circ[X(t)]^{-1}$, is unique with the following properties:
\begin{enumerate}
  \item $L^{2}$-boundedness: $\sup\limits_{t\in\mathbb{R}}\int_{H}\|x\|_{H}^{2}\widehat{\mu}(t)(\d x)<+\infty$;
  \item Flow property: $\mu(t,s,\widehat{\mu}(s))=\widehat{\mu}(t)$~for~all~$t\geq s$.
\end{enumerate}
Recalling that $\mu(t,s,\mu_{0})$ denotes the distribution of $X(t,s,\zeta_{s})$, with
$\mu_{0}=P\circ\zeta_{s}^{-1}$.
\end{theorem}

\begin{proof}
For any fixed interval $[a,b]\subset\mathbb{R}$ we denote
$$K:=L^{\alpha}([a,b]\times\Omega,\d t\otimes P;V), \quad K^{*}:=
L^{\frac{\alpha}{\alpha-1}}([a,b]\times\Omega,\d t\otimes P;V^{*}).$$
Since $K$ is reflexive, we may assume, going if necessary to a subsequence, that

\begin{enumerate}
  \item[(1)] $X(\cdot,-n,0)\rightarrow X(\cdot)$  in $L^{2}([a,b]\times\Omega,\d t\otimes P;H)$ and
  $ \overline{X}(\cdot,-n,0)\rightarrow \overline{X}(\cdot)$ weakly in $K$;
  \item[(2)] $A(\cdot,\overline{X}(\cdot,-n,0))\rightarrow Y(\cdot)$ weakly in $K^{*}$.
\end{enumerate}
Thus for all $v\in V$, $\varphi\in L^{\infty}([a,b]\times\Omega)$ by Fubini's theorem we get
\begin{align*}
   & E\int_{a}^{b}~_{V^{*}}\langle X(t),\varphi(t)v\rangle_{V}\d t \\
   & =\lim_{n\rightarrow\infty}E\int_{a}^{b}~_{V^{*}}\langle X(t,-n,0),\varphi(t)v\rangle_{V}\d t\\
   & =\lim_{n\rightarrow\infty}E\left(\int_{a}^{b}~_{V^{*}}\langle X(a,-n,0)
     +\int_{a}^{t}A(\sigma,\overline{X}(\sigma,-n,0))\d\sigma
      +\int_{a}^{t}B(\sigma)\d W(\sigma),\varphi(t)v\rangle_{V}\d t\right)\\
   & =E\left(\int_{a}^{b}~_{V^{*}}\langle X(a)+\int_{a}^{t}Y(\sigma)\d\sigma
      +\int_{a}^{t}B(\sigma)\d W(\sigma),\varphi(t)v\rangle_{V}\d t\right).
\end{align*}
It follows that
\begin{equation*}
  X(t)=X(a)+\int_{a}^{t}Y(\sigma)\d\sigma+\int_{a}^{t}B(\sigma)\d W(\sigma),\quad\d t\otimes P\rm\mbox{-}a.e.
\end{equation*}
Thus, it remains to verify that
$$Y=A(\cdot,\overline{X}),\quad \d t\otimes P\rm\mbox{-}a.e. $$
To this end, for any $\phi\in K\cap L^{2}([a,b]\times\Omega,\d t\otimes P;H)$ we have
\begin{align}\label{000}
   & E\left(\|X(t,-n,0)\|_{H}^{2}-\|X(a,-n,0)\|_{H}^{2}\right) \\\nonumber
   & =E\left(\int_{a}^{t}\left(2_{V^{*}}\langle A(\sigma,\overline{X}(\sigma,-n,0)),
     \overline{X}(\sigma,-n,0)\rangle_{V}+\|B(\sigma)\|_{L_{2}(U,H)}^{2}\right)\d\sigma\right)\\\nonumber
   & =E\bigg(\int_{a}^{t}\Big(2_{V^{*}}\langle A(\sigma,\overline{X}(\sigma,-n,0))
     -A(\sigma,\phi(\sigma)),\overline{X}(\sigma,-n,0)-\phi(\sigma)\rangle_{V}\\\nonumber
   & \qquad+2_{V^{*}}\langle A(\sigma,\phi(\sigma)),\overline{X}(\sigma,-n,0)\rangle_{V}
    +\|B(\sigma)\|_{L_{2}(U,H)}^{2}\\\nonumber
   & \qquad+2_{V^{*}}\langle A(\sigma,\overline{X}(\sigma,-n,0))-A(\sigma,\phi(\sigma)),
     \phi(\sigma)\rangle_{V}\Big)\d\sigma\bigg)\\\nonumber
   & \leq E\bigg(\int_{a}^{t}\Big(2_{V^{*}}\langle A(\sigma,\overline{X}(\sigma,-n,0))
     -A(\sigma,\phi(\sigma)),\phi(\sigma)\rangle_{V}\\\nonumber
   & \qquad +2_{V^{*}}\langle A(\sigma,\phi(\sigma)),\overline{X}(\sigma,-n,0)\rangle_{V}
     +\|B(\sigma)\|_{L_{2}(U,H)}^{2}\Big)\d\sigma\bigg).
\end{align}

 Similar to the proof of \eqref{Bountheq2}, for given nonnegative
$\psi\in L^{\infty}([0,T],\d t;\mathbb{R})$,
first multiplying $\psi(t)$ on both sides of (5.6), then integrating with respect to $t$
from $a$ to $b$ and letting $n\rightarrow\infty$, we have
\begin{align}\label{111}
   & E\int_{a}^{b}\psi(t)\left(\|X(t)\|_{H}^{2}-\|X(a)\|_{H}^{2}\right)\d t \\\nonumber
   & \leq E\bigg(\int_{a}^{b}\psi(t)\int_{a}^{t}\Big(2_{V^{*}}\langle Y(\sigma)
     -A(\sigma,\phi(\sigma)),\phi(\sigma)\rangle_{V}\\\nonumber
   & \qquad+2_{V^{*}}\langle A(\sigma,\phi(\sigma)),\overline{X}(\sigma)\rangle_{V}
     +\|B(\sigma)\|_{L_{2}(U,H)}^{2}\Big)\d\sigma \d t\bigg).
\end{align}
In view of the product rule, we obtain
\begin{align}\label{222}
   & E\int_{a}^{b}\psi(t)\left(\|X(t)\|_{H}^{2}-\|X(a)\|_{H}^{2}\right)\d t \\\nonumber
   & =E\left(\int_{a}^{b}\psi(t)\int_{a}^{t}\left(2_{V^{*}}\langle Y(\sigma),\overline{X}(\sigma)\rangle_{V}
     +\|B(\sigma)\|_{L_{2}(U,H)}^{2}\right)\d\sigma \d t\right).
\end{align}
Therefore, \eqref{111} and \eqref{222} imply
\begin{equation}\label{ANine5}
  E\left(\int_{a}^{b}\psi(t)\int_{a}^{t}2_{V^{*}}\langle Y(\sigma)-A(\sigma,\phi(\sigma)),
  \overline{X}(\sigma)-\phi(\sigma)\rangle_{V}\d\sigma \d t\right)\leq0.
\end{equation}
Similar to the proof of Theorem \ref{Boundedth}, we obtain that
$Y=A(\cdot,\overline{X}),~\d t\otimes P$-a.e. This completes the existence proof, i.e.
\begin{equation*}
  X(t)=X(a)+\int_{a}^{t}A(\sigma,\overline{X}(\sigma))\d\sigma
  +\int_{a}^{t}B(\sigma)\d W(\sigma), \quad \d t\otimes P\rm\mbox{-}a.e.
\end{equation*}
By the arbitrariness of interval $[a,b]\subset \mathbb{R}$, we conclude that $X(\cdot)$ is a solution on $\mathbb{R}$.
It follows from Lemma \ref{anbound} that $\sup\limits_{t\in\mathbb{R}}E\|X(t)\|_{H}^{2}<\infty$.

Now we prove the uniqueness of $L^{2}$-bounded solution.
Let $X(\cdot)$, $Y(\cdot)$ be two $L^{2}$-bounded solutions, then by \eqref{ANSCP} we have
\begin{align*}
   & E\|X(t)-Y(t)\|_{H}^{2} \\
   & =E\|X(t,-n,X(-n))-Y(t,-n,Y(-n))\|_{H}^{2}\\
   & \leq\left\{\frac{\lambda}{2}(r-2)(t+n)\right\}^{-\frac{2}{r-2}}\rightarrow 0
   \quad {\rm{as}} ~ n\rightarrow\infty.
\end{align*}

Finally, we show that $\widehat{\mu}$ is unique with properties (i) and (ii).
Note that
\begin{equation*}
  \sup_{t\in\mathbb{R}}\int_{H}\|x\|_{H}^{2}\widehat{\mu}(t)(\d x)=
  \sup_{t\in\mathbb{R}}E\|X(t)\|_{H}^{2}<\infty.
\end{equation*}
According to the Chapman-Kolmogorov equation, we have
\begin{equation}\label{CKeq}
  \mu(t,s,\L(X(s,-n,0)))=\L(X(t,-n,0)).
\end{equation}
In view of \eqref{ANSCP}, we get
\begin{align*}
   & E\|X(t,s,X(s))-X(t,s,X(s,-n,0))\|_{H}^{2}\leq E\|X(s)-X(s,-n,0)\|_{H}^{2}.
\end{align*}
This inequality and \eqref{CKeq} yield
\begin{equation*}
  \mu(t,s,\widehat{\mu}(s))=\widehat{\mu}(t).
\end{equation*}
It remains to prove the uniqueness of $\widehat{\mu}$.
Let $\mu_{1}$, $\mu_{2}$ be two mappings which satisfy (i) and (ii), and let $\zeta_{n,1}$, $\zeta_{n,2}$
be random variables with distributions $\mu_{1}(-n)$, $\mu_{2}(-n)$ respectively. Consider the solutions
$X(t,-n,\zeta_{n,1})$, $X(t,-n,\zeta_{n,2})$ on $[-n,\infty)$, then we have
\begin{align*}
   & d_{BL}(\mu_{1}(t),\mu_{2}(t)) \\
   & =d_{BL}(\mu(t,-n,\mu_{1}(-n)),\mu(t,-n,\mu_{2}(-n)))\\
   & =\sup_{\|f\|_{BL}\leq1}\left|\int_{H}f(x)\d(\mu(t,-n,\mu_{1}(-n))-\mu(t,-n,\mu_{2}(-n)))\right|\\
   & =\sup_{\|f\|_{BL}\leq1}\left|\int_{H}[f(X(t,-n,\zeta_{n,1}))-f(X(t,-n,\zeta_{n,2}))]\d P\right|\\
   & \leq\left(E\|X(t,-n,\zeta_{n,1})-X(t,-n,\zeta_{n,2})\|_{H}^{2}\right)^{\frac{1}{2}}\\\
   & \leq\left\{\frac{\lambda}{2}(r-2)(t+n)\right\}^{-\frac{1}{r-2}}\rightarrow 0\quad{\rm{as}}
     ~ n\rightarrow \infty.
\end{align*}
Thus, $\mu_{1}(t)=\mu_{2}(t)$ for all $t\in\R$.
\end{proof}

Completely similar to Proposition \ref{CBprop}, Theorems \ref{Pth}, \ref{APth} and
\ref{AAth}, we can get the following proposition and theorems.

\begin{prop}\label{ANCBprop}
Consider equation \eqref{spdean}.
Suppose that $A$, $B$, $A_{n}$, $B_{n}$ satisfy {\rm{(H1)}}, {\rm{(H2$''$)}}, {\rm{(H3)}} and {\rm{(H4)}}
with the same constants $\lambda$, $r$, $c_{1}$, $c_{2}$, $c_{3}$, $M_{0}$, $\alpha$.
Let $X(\cdot)$, $X_{n}(\cdot)$ be the $L^{2}$-bounded solutions of equation \eqref{spdean}
corresponding to $A$, $B$ and $A_{n}$, $B_{n}$ respectively. Assume in addition that

\begin{enumerate}
  \item $\lim\limits_{n\rightarrow\infty}\|A_{n}(t,x)-A(t,x)\|_{V^{*}}=0$ for all $x\in V$, $t\in \mathbb{R}$;
  \item $\lim\limits_{n\rightarrow\infty}\|B_{n}(t)-B(t)\|_{L_{2}(U,H)}=0$ for all $x\in V$, $t\in \mathbb{R}$;
  \item for each $t\in \mathbb{R}$ the family of distributions
        $\{P\circ[X_{n}(t)]^{-1}\}_{n\in\mathbb{N}}$ is tight.
\end{enumerate}
Then
$$\lim_{n\rightarrow\infty}d_{BL}(\mathcal{L}(X_{n}),\mathcal{L}(X))=0\quad {\rm{in}}~Pr(C(\mathbb{R},H)).$$

In particular,
$$\lim_{n\rightarrow\infty}d_{BL}(\mathcal{L}(X_{n}(t+\cdot)),\mathcal{L}(X(t+\cdot)))=0
 \quad {\rm{in}}~Pr(C(\mathbb{R},H))~{\rm{for~all}}~t\in\mathbb{R}.$$
\end{prop}

\begin{theorem}\label{ANPth}
Consider equation \eqref{spdean}.
Suppose that {\rm{(H1)}}, {\rm{(H2$''$)}}, {\rm{(H3)}} and {\rm{(H4)}} hold. Assume further that
the mappings $A$ and $B$ are $T$-periodic in t.
Then the unique $L^{2}$-bounded solution
is $T$-periodic in distribution.

In particular, this unique $L^{2}$-bounded solution is stationary
provided the mappings $A$ and $B$ are independent of t.
\end{theorem}

\begin{theorem}\label{ANAPth}
Consider equation \eqref{spdean}.
Suppose that {\rm{(H1)}}, {\rm{(H2$''$)}}, {\rm{(H3)}}, {\rm{(H4)}} and {\rm{(H5)}} hold.
Assume further that the coefficients $A$ and $B$ are uniformly almost periodic (uniformly almost automorphic).
Then the unique $L^{2}$-bounded solution is almost periodic in distribution
(almost automorphic in distribution).
\end{theorem}

\section{Stability of the bounded solution}
In this section, we prove that the $L^{2}$-bounded solutions of equations \eqref{eqSPDE1}
and \eqref{spdean} are globally asymptotically stable.

\begin{definition}[See Fu and Liu \cite{FL}]\rm
We say that a solution $X(\cdot)$ of equation \eqref{eqSPDE1} or \eqref{spdean} is {\em stable
in square-mean sense}, if for each $\epsilon>0$, there exists $\delta>0$ such that for all $t\geq0$
\begin{equation*}
  E\|X(t,0,\zeta_{0})-X(t)\|_{H}^{2}<\epsilon,
\end{equation*}
whenever $E\|\zeta_{0}-X(0)\|_{H}^{2}<\delta$. The solution $X(\cdot)$ is said to be {\em asymptotically
stable in square-mean sense} if it is stable in square-mean sense and
\begin{equation}\label{stab}
  \lim_{t\rightarrow\infty}E\|X(t,0,\zeta_{0})-X(t)\|_{H}^{2}=0.
\end{equation}
We say $X(\cdot)$ is {\em globally asymptotically stable in square-mean sense} provided  \eqref{stab} holds
for any $\zeta_{0}\in L^{2}(\Omega,\mathcal{F}_{0},P;H)$.
\end{definition}

\begin{theorem}\label{gasms}
Consider equation \eqref{eqSPDE1}.
Suppose that {\rm{(H1)}}, {\rm{(H2$'$)}}, {\rm{(H3)}} and {\rm{(H4)}} hold,
then the unique $L^{2}$-bounded solution of equation \eqref{eqSPDE1} is
globally asymptotically stable in square-mean sense. Moreover, for any $t\geq s$ and
$\zeta_{s}\in L^{2}(\Omega,\mathcal{F}_{s},P;H)$, we have
\begin{equation*}
  E\|X(t,s,\zeta_{s})-X(t)\|_{H}^{2}\leq {\rm{e}}^{-\lambda(t-s)}E\|\zeta_{s}-X(s)\|_{H}^{2}.
\end{equation*}
\end{theorem}

\begin{proof}
In view of It\^o's formula, the product rule and (H2$'$), we have
\begin{align*}
  & E\left({\rm{e}}^{\lambda(t-s)}\|X(t,s,\zeta_{s})-X(t)\|_{H}^{2}\right)\\
  & =E\|\zeta_{s}-X(s)\|_{H}^{2}+\int_{s}^{t}\lambda {\rm{e}}^{\lambda(\sigma-s)}E\|X(\sigma,s,\zeta_{s})-X(\sigma)\|_{H}^{2}\d\sigma \\
  & \quad +E\int_{s}^{t}{\rm{e}}^{\lambda(\sigma-s)}\Big(2_{V^{*}}\langle A(\sigma,\overline{X}(\sigma,s,\zeta_{s}))
    -A(\sigma,\overline{X}(\sigma)),\overline{X}(\sigma,s,\zeta_{s})-\overline{X}(\sigma)\rangle_{V}\\
  & \qquad +\|B(\sigma,\overline{X}(\sigma,s,\zeta_{s}))-B(\sigma,\overline{X}(\sigma))\|_{L_{2}(U,H)}^{2}\Big)\d\sigma\\
  & \leq E\|\zeta_{s}-X(s)\|_{H}^{2}.
\end{align*}
It follows that
\begin{equation*}
  E\|X(t,s,\zeta_{s})-X(t)\|_{H}^{2}\leq {\rm{e}}^{-\lambda(t-s)}E\|\zeta_{s}-X(s)\|_{H}^{2},\quad \rm{for~all}~t\geq s.
\end{equation*}
The proof is complete.
\end{proof}

Applying Lemma \ref{adcont} we obtain the following result:
\begin{theorem}
Consider equation \eqref{spdean}.
Suppose that {\rm{(H1)}}, {\rm{(H2$''$)}}, {\rm{(H3)}} and {\rm{(H4)}} hold,
then the unique $L^{2}$-bounded solution of equation \eqref{spdean} is globally asymptotically stable
in square-mean sense. Moreover, for any $t\geq s$ and
$\zeta_{s}\in L^{2}(\Omega,\mathcal{F}_{s},P;H)$, we have
\begin{equation*}
  E\|X(t,s,\zeta_{s})-X(t)\|_{H}^{2}\leq E\|\zeta_{s}-X(s)\|_{H}^{2}\wedge
  \left\{\frac{\lambda}{2}(r-2)(t-s)\right\}^{-\frac{2}{r-2}}.
\end{equation*}
\end{theorem}

\section{Applications}
In this section, we illustrate our theoretical results by two examples. For simplicity,
we mainly consider the additive type noise in these examples.

\subsection{Stochastic reaction diffusion equations.}
Let $\Lambda$ be an open bounded subset of $\mathbb{R}^{n}$, $n\in\mathbb{N}$. Consider the equation
\begin{equation}\label{ex1}
 \d u=\left(\Delta u-au|u|^{p-2}+\phi(t)u\right)\d t+B(t)\d W(t),
\end{equation}
where $W(\cdot)$ is a two-sided cylindrical $Q$-Wiener process with $Q=I$ on a separable Hilbert space
$(U,\langle~,~\rangle_{U})$ and $p\in[2,\infty)$. Here $a>0$ is a constant and
$\phi(\cdot)$ is  bounded, i.e. there exists a constant $C_{1}>0$ such that
$|\phi(t)|\leq C_{1}$ for all $t\in\R$.
We define $V_{1}:=H_{0}^{1,2}(\Lambda)$, $V_{2}:=L^{p}(\Lambda)$, $H:=L^{2}(\Lambda)$,
$V_{1}^{*}:=(H_{0}^{1,2}(\Lambda))^{*}$, $V_{2}^{*}:=(L^{p}(\Lambda))^{*}$,
$V:=V_{1}\cap V_{2}$ and
$$A_{1}(t,u):=\Delta u+\phi(t)u ,\quad A_{2}(u):=-au|u|^{p-2}.$$

\begin{theorem}\label{app1}
Let $\lambda_{*}$ be the first eigenvalue of $-\Delta$ with the Dirichlet boundary condition and
assume that $\lambda_{*}-C_1>0$.

{\rm{(1)}} If $\|B(t)\|^{2}_{L_{2}(U,H)}\leq M$ for some constant $M>0$, then there exists a unique $L^{2}$-bounded solution $X(\cdot)$ to
equation \eqref{ex1}, which is globally asymptotically stable in square-mean sense. Furthermore, $X(\cdot)$ is $T$-periodic
in distribution (stationary) if $B$ and $\phi$ are $T$-periodic (independent of $t$).

{\rm{(2)}} Let $S=H_{0}^{1,2}(\Lambda)$. Suppose that there exists a constant $\hat M>0$
such that $\|B(t)\|^{2}_{L_{2}(U,S)}\leq \hat M$. Then the $L^{2}$-bounded solution $X(\cdot)$
is almost periodic (almost automorphic) in distribution if $B$ and $\phi$ are almost periodic (almost automorphic).
\end{theorem}

\begin{proof}
(1) In order to prove (1), by Theorems \ref{Boundedth}, \ref{gasms} and \ref{Pth},
it suffices to show that $A$ and $B$ satisfy (H1), (H2$'$), (H3) and (H4).

(H1) $A_{1}$ is obviously hemicontinuous.
We now prove that $A_{2}$ is hemicontinuous.
Let $u$, $v$, $w\in V$. For $\theta\in\mathbb{R}$,
without loss of generality, we assume $|\theta|\leq1$, then we have
\begin{align}\label{appien1}
   & ~_{V^{*}}\langle A_{2}(u+\theta v)-A_{2}(u),w\rangle_{V} \\\nonumber
   & =\int_{\Lambda}\left(-\left(u(\xi)+\theta v(\xi)\right)|u(\xi)+\theta v(\xi)|^{p-2}w(\xi)
     +u(\xi)|u(\xi)|^{p-2}w(\xi)\right)\d\xi\\\nonumber
   & \leq\int_{\Lambda}\left(2^{p-2}\left(|u(\xi)|^{p-1}+|v(\xi)|^{p-1}\right)|w(\xi)|
     +|u(\xi)|^{p-1}|w(\xi)|\right)\d\xi<\infty.
\end{align}
The last inequality holds since $u$, $v$, $w\in L^{p}(\Lambda)$. Then
$~_{V^{*}}\langle A_{2}(u+\theta v)-A_{2}(u),w\rangle_{V}$ converges to zero
as $\theta\rightarrow0$ by Lebesgue's dominated convergence theorem. So, (H1) holds.

(H2$'$) If $u,v\in V$ then there exist $u_{n},~v_{n}\in C_{c}^{\infty}(\Lambda)$, $n\in \mathbb{N}$
such that $u_{n}\rightarrow u$, $v_{n}\rightarrow v$ as $n\rightarrow \infty$ in $V_{1}$.
Hence we get
\begin{align*}
   & ~_{V_{1}^{*}}\langle A_{1}(t,u)-A_{1}(t,v),u-v\rangle_{V_{1}} \\
   & =\lim_{n\rightarrow \infty}~_{V_{1}^{*}}\langle\Delta u_{n}-\Delta v_{n},u_{n}-v_{n}\rangle_{V_{1}}
      +\phi(t)\langle u-v,u-v\rangle_{H}\\
   & \leq-\lambda_{*}\lim_{n\rightarrow \infty}\langle u_{n}- v_{n},u_{n}-v_{n}\rangle_{H}
     +\phi(t)\|u-v\|_{H}^{2}\\
   & \leq-(\lambda_{*}-C_1)\|u-v\|_{H}^{2}
\end{align*}
and
\begin{equation*}
  ~_{V_{2}^{*}}\langle A_{2}(u)-A_{2}(v),u-v\rangle_{V_{2}}=-a\int_{\Lambda}
  \left(u(\xi)|u(\xi)|^{p-2}-v(\xi)|v(\xi)|^{p-2}\right)\left(u(\xi)-v(\xi)\right)\d\xi\leq0.
\end{equation*}
Then
\begin{equation*}
  2_{V^{*}}\langle A(t,u)-A(t,v),u-v\rangle_{V}+\|B(t)-B(t)\|^{2}_{L_{2}(U,H)}
  \leq-2(\lambda_{*}-C_1)\|u-v\|^{2}_{H}.
\end{equation*}
So (H2$'$) holds with $\lambda=2(\lambda_{*}-C_1)>0$.

(H3) For all $v\in V$, $t\in\mathbb{R}$ we have
\begin{align*}
  ~_{V_{1}^{*}}\langle A_{1}(t,v),v\rangle_{V_{1}} & = \lim_{n\rightarrow \infty}
     ~_{V_{1}^{*}}\langle\Delta v_{n},v_{n}\rangle_{V_{1}}+\phi(t)\langle v,v\rangle_{H}\\
   & =-\int_{\Lambda}|\nabla v(\xi)|^{2}\d\xi+\phi(t)\|v\|_{H}^{2}\\
   & =\|v\|_{H}^{2}-\|v\|^{2}_{1,2}+\phi(t)\|v\|^{2}_{H}\\
   & \leq\left(C_1+1\right)\|v\|^{2}_{H}-\|v\|^{2}_{V_{1}}
\end{align*}
and
\begin{equation*}
 ~_{V_{2}^{*}}\langle A_{2}(v),v\rangle_{V_{2}}= -a\int_{\Lambda}|v(\xi)|^{p}\d\xi =-a\|v\|^{p}_{V_{2}}.
\end{equation*}
Then
\begin{equation*}
  2_{V^{*}}\langle A(t,v),v\rangle_{V}+\|B(t)\|_{L_{2}(U,H)}^{2}
    \leq2\left(C_1+1\right)\|v\|_{H}^{2}-2\|v\|_{V_{1}}^{2}-2a\|v\|_{V_{2}}^{p}+M.
\end{equation*}
So (H3) holds with $\alpha_{1}=2$, $\alpha_{2}=p$.

(H4) For all $u$, $v\in V$, $t\in\R$ we have
\begin{align*}
  \left|_{V_{1}^{*}}\langle A_{1}(t,u),v\rangle_{V_{1}}\right| & =\left|\lim_{n\rightarrow \infty}
     ~_{V_{1}^{*}}\langle\Delta u_{n},v_{n}\rangle_{V_{1}}+\phi(t)\langle u,v\rangle_{H}\right| \\
   & \leq\|\nabla u\|_{H}\|\nabla v\|_{H}+C\|u\|_{H}\|v\|_{H}\\
   & \leq\left(C_1+1\right)\|u\|_{V_{1}}\|v\|_{V_{1}}
\end{align*}
and
\begin{equation*}
   \left|_{V_{2}^{*}}\langle A_{2}(u),v\rangle_{V_{2}}\right|  =\left|a\int_{\Lambda}-u(\xi)|u(\xi)|^{p-2}v(\xi)\d\xi\right|
    \leq a\|u\|_{V_{2}}^{p-1}\|v\|_{V_{2}}.
\end{equation*}
Therefore, we get
\[
\|A_{1}(t,u)\|_{V_{1}^{*}}\leq\left(C_1+1\right)\|u\|_{V_{1}},\quad \|A_{2}(u)\|_{V_{2}^{*}}
\leq a\|u\|^{p-1}_{V_{2}}.
\]
So (H4) holds with $\alpha_{1}=2$, $\alpha_{2}=p$.

(2) Note that  $\|B(t)\|^{2}_{L_{2}(U,H)}\leq\|B(t)\|^{2}_{L_{2}(U,S)}$ and
$\|B(t)\|^{2}_{L_{2}(U,H_{n})}\leq\|B(t)\|^{2}_{L_{2}(U,S)}$ for all $t\in\R$ (see, e.g. \cite[Remark B.0.6]{PR}).
So, in order to prove the almost periodic (almost automorphic) property of the $L^2$-bounded solution, by Theorems
\ref{APth} and \ref{AAth}, it suffices to show that (H5) holds. To this end,
we define $T_{n}=-\Delta\left(I-\frac{\Delta}{n}\right)^{-1}=n(I-(I-\frac{\Delta}{n})^{-1})$
which is the Yosida approximation of $\Delta$.
Note that $T_{n}$ are continuous on $W_{0}^{1,2}(\Lambda)$. Since the heat semigroup
$\{P_{t}\}_{t\geq0}$ (generated by $\Delta$) is contractive on $L^{p}(\Lambda)$,
$p>1$ (see Theorem 3.6 on page 215 of \cite{Pazy}) and
$(I-\frac{\Delta}{n})^{-1}u=\int_{0}^{\infty}{\rm{e}}^{-t}P_{\frac{t}{n}}u\d t$,
$T_{n}$ are continuous on $L^{p}(\Lambda)$.

For all $u\in V$, $t\in\R$ we have
\begin{align*}
  ~_{V_{1}^{*}}\langle\Delta u,T_{n}u\rangle_{V_{1}} & = \lim_{m\rightarrow \infty}
     ~_{V_{1}^{*}}\langle\Delta u_{m},T_{n}u_{m}\rangle_{V_{1}}\\
   & =\lim_{m\rightarrow \infty}\langle\sum_{j=1}^{\infty}
     \langle\Delta u_{m},e_{j}\rangle_{H}e_{j},T_{n}u_{m}\rangle_{H}\\
   & =\lim_{m\rightarrow \infty}\langle\sum_{j=1}^{\infty}
     -\lambda_{j}\langle u_{m},e_{j}\rangle_{H}e_{j},T_{n}u_{m}\rangle_{H}\\
   & \leq-\lambda_{*}\lim_{m\rightarrow \infty}\langle u_{m},T_{n}u_{m}\rangle_{H}\\
   & \leq-\lambda_{*}\|u\|_{n}^{2}
\end{align*}
and
\[
\phi(t)\langle u,T_{n}u\rangle_{H}=\phi(t)\|u\|_{n}^{2}\leq C_1\|u\|_{n}^{2}.
\]
In view of the contractivity of $\{P_{t}\}_{t\geq0}$ on $L^{p}(\Lambda)$, we have
\begin{equation*}
  _{V_{2}^{*}}\langle A_{2}(u),T_{n}u\rangle_{V_{2}}
  =n\int_{0}^{\infty}{\rm{e}}^{-t}\left(\int_{\Lambda}-au(\xi)|u(\xi)|^{p-2}\left(u(\xi)
  -P_{\frac{t}{n}}u(\xi)\right)d\xi\right)\d t\leq0.
\end{equation*}
Then we obtain
\begin{equation*}
  2_{V^{*}}\langle A(t,u),T_{n}u\rangle_{V}+\|B(t)\|^{2}_{L_{2}(U,H_{n})}\leq
  -2\left(\lambda_{*}-C_1\right)\|u\|_{n}^{2}+\hat M.
\end{equation*}
That is, (H5) holds.

The proof is complete.
\end{proof}

\begin{remark}\rm
\begin{enumerate}
  \item In the above stochastic reaction diffusion equation, we can also consider multiplicative noise case.
Here we give a simple example; see e.g. \cite{LiuW101}. Consider
\begin{equation}\label{lmneq}
  B(t,v)u:=B_{0}(t)u+\sum_{i=1}^{N}\phi_{i}(t)\langle u,u_{i}\rangle_{U}v,~u\in U, ~v\in V
\end{equation}
where $B_{0}:\mathbb{R}\rightarrow L_{2}(U,S)$ is progressively measurable,
$u_{i}\in U$, $\phi_{i}:\mathbb{R}\rightarrow\mathbb{R}$ and there exist  constants $C_{1,i}>0$
such that $|\phi_{i}(t)|\leq C_{1,i}$ for all $t\in\R$, $i=1,2,...,N$.
Suppose that there exists a constant $\hat{M}>0$ such that $\|B_{0}(t)\|^{2}_{L_{2}(U,S)}\leq \hat{M}$
for all $t\in\R$ and
$$\lambda_{*}-C_1-\frac{N+1}{2}\sum\limits_{i=1}^{N}C_{1,i}^{2}\|u_{i}\|_{U}^{2}>0.$$
Then there exists a unique $L^{2}$-bounded solution $X(\cdot)$, which is globally
asymptotically stable in square-mean sense. Furthermore, if
 $\phi$, $\phi_{i}$, $i=1,2,...,N$ and $B_{0}$ are $T$-periodic (respectively, almost periodic,
almost automorphic), then the $L^{2}$-bounded solution $X(\cdot)$ is $T$-periodic
(respectively, almost periodic, almost automorphic) in distribution. In particular,
$X(\cdot)$ is stationary provided $\phi$, $\phi_{i}$, $i=1,2,...,N$ and $B_0$ are independent of t.
  \item Note that when the noise is additive or of linear form as in \eqref{lmneq},
the main result of \cite{GaoP} is a special case of the above example for $p=4$.
\end{enumerate}
\end{remark}

\subsection{Stochastic porous media equations}
Let $\Lambda$ be an open bounded subset of $\mathbb{R}^{n}$, $n\in\mathbb{N}$. Consider the equation
\begin{equation}\label{spme}
  \d u=\left(\Delta(|u|^{p-2}u)+\phi(t)u\right)\d t+B(t)\d W(t),
\end{equation}
where $W(\cdot)$ is a two-sided cylindrical $Q$-Wiener process with $Q=I$ on a separable Hilbert space
$(U,\langle~,~\rangle_{U})$, $p>2$. And there exist constants $C_{1}>C_{2}>0$
such that $-C_{1}<\phi(t)<-C_{2}$ for all $t\in\R$. We define
$$V:=L^{p}(\Lambda)\subset H:=W_{0}^{-1,2}(\Lambda)\subset V^{*}.$$

\begin{theorem}

{\rm{(1)}} If $\|B(t)\|^{2}_{L_{2}(U,H)}\leq M$ for some constant $M > 0$, then there exists a unique $L^{2}$-bounded
solution $X(\cdot)$ to equation \eqref{spme}, which is globally asymptotically stable in square-mean sense.
Furthermore, $X(\cdot)$ is $T$-periodic in distribution (stationary) if $\phi$ and $B$ are T-periodic (independent of t).

{\rm{(2)}} Let $S=L^{2}(\Lambda)$. Suppose that there exists a constant $\hat{M}>0$ such that
$\|B(t)\|^{2}_{L_{2}(U,S)}\leq \hat{M}$. Then the $L^{2}$-bounded solution $X(\cdot)$ is almost periodic
(almost automorphic) in distribution provided $\phi$ and $B$ are almost periodic (almost automorphic).
\end{theorem}

\begin{proof}
Fix $u\in V$, $t\in\R$, for all $v\in V$ we denote
\begin{equation*}
  ~_{V^{*}}\langle A(t,u),v\rangle_{V}:=-\int_{\Lambda}u(\xi)|u(\xi)|^{p-2}v(\xi)\d\xi
  +\int_{\Lambda}\phi(t) u(\xi)v(\xi)\d\xi.
\end{equation*}
We first show that $A:\R\times V\rightarrow V^{*}$ is well-defined.
Indeed, for all $u$, $v\in V$, $t\in\R$
\begin{align*}
   & \left|_{V^{*}}\langle A(t,u),v\rangle_{V}\right| \\
   & \leq\int_{\Lambda}|u(\xi)|^{p-1}|v(\xi)|\d\xi+C_{1}\int_{\Lambda}|u(\xi)||v(\xi)|\d\xi\\
   & \leq\left(\int_{\Lambda}|u(\xi)|^{p}\d\xi\right)^{\frac{p-1}{p}}
     \left(\int_{\Lambda}|v(\xi)|^{p}\d\xi\right)^{\frac{1}{p}}
     +C_{1}\left(\int_{\Lambda}|u(\xi)|^{\frac{p}{p-1}}\d\xi\right)^{\frac{p-1}{p}}
     \left(\int_{\Lambda}|v(\xi)|^{p}\d\xi\right)^{\frac{1}{p}}\\
   & \leq\|u\|_{L^{p}}^{p-1}\|v\|_{L^{p}}+C_{1}\|u\|_{L^{\frac{p}{p-1}}}\|v\|_{L^{p}}\\
   & \leq \|u\|_{L^{p}}^{p-1}\|v\|_{L^{p}}+C_{1}\left(\frac{1}{p-1}\|u\|_{L^{p}}^{p-1}
     +\frac{p-2}{p-1}\left(|\Lambda|\right)^{\frac{p-1}{p}}\right)\|v\|_{L^{p}}.
\end{align*}
Therefore, $A:\R\times V\rightarrow V^{*}$ is well-defined and we have
\begin{equation}\label{AP3ine1}
  \|A(t,u)\|_{V^{*}}\leq\left(1+C_{1}\frac{1}{p-1}\right)\|u\|_{L^{p}}^{p-1}
  +\frac{p-2}{p-1}\left(|\Lambda|\right)^{\frac{p-1}{p}}C_{1}.
\end{equation}

Next we verify assertions (1) and (2).

(1) It suffices to show that (H1), (H2$''$), (H3) and (H4) hold.

(H1) follows immediately from \eqref{appien1}.

(H2$''$) For all $u$, $v\in V$, $t\in\R$ we have
\begin{align*}
   & ~_{V^{*}}\langle A(t,u)-A(t,v),u-v\rangle_{V} \\
   & =-\langle u|u|^{p-2}-v|v|^{p-2},u-v\rangle_{L^{2}}+\phi(t)\|u-v\|_{L^{2}}^{2}\\
   & \leq-2^{2-p}\|u-v\|_{L^{p}}^{p}+\phi(t)\|u-v\|_{L^{2}}^{2}\\
   & \leq-2^{2-p}\|u-v\|_{H}^{p}.
\end{align*}
Therefore, (H2$''$) holds with $r=p$, $\lambda=2^{3-p}$.

(H3)  Note that for all $u\in V$, $t\in\R$
\begin{equation*}
  ~_{V^{*}}\langle A(t,u),u\rangle_{V}=-\int_{\Lambda}u(\xi)|u(\xi)|^{p-2}u(\xi)\d\xi
  +\phi(t)\int_{\Lambda}u(\xi)u(\xi)\d\xi\leq-\|u\|_{L^{p}}^{p},
\end{equation*}
so we have
\begin{equation*}
  2_{V^{*}}\langle A(t,u),u\rangle_{V}+\|B(t)\|_{L_{2}(U,H)}^{2}\leq-2\|u\|_{L^{p}}^{p}+M.
\end{equation*}
That is, (H3) holds with $\alpha=p$.

(H4) holds by (\ref{AP3ine1}) with $\alpha=p$.

(2) Like the proof of Theorem \ref{app1}, it suffices to verify (H5). Let $S=L^{2}(\Lambda)$ and $\Delta$ be the Laplace operator
on $L^{2}(\Lambda)$ with the Dirichlet boundary condition. We define
$T_{n}=-\Delta\left(I-\frac{\Delta}{n}\right)^{-1}=n\left(I-(I-\frac{\Delta}{n})^{-1}\right)$.
The continuity of $T_{n}$ on $L^{p}(\Lambda)$, $p>1$ was already shown in the proof of Theorem \ref{app1}.
Then recalling the contractivity of $\{P_{t}\}_{t\geq0}$ on $L^{p}(\Lambda)$, $p>1$, we obtain
\begin{align*}
   & ~_{V^{*}}\langle \Delta(u|u|^{p-2})+\phi(t) u,-\Delta(I-\frac{\Delta}{n})^{-1}u\rangle_{V} \\
   & =-\langle u|u|^{p-2},nu-n(I-\frac{\Delta}{n})^{-1}u\rangle_{L^{2}}+\phi(t)\|u\|_{n}^{2}\\
¡¡ & =-\langle u|u|^{p-2},nu-n\int_{0}^{\infty}{\rm{e}}^{-t}P_{\frac{t}{n}}u(\xi)\d t\rangle_{L^{2}}
     +\phi(t)\|u\|_{n}^{2}\\
   & =-n\int_{0}^{\infty}{\rm{e}}^{-t}\left(\int_{\Lambda}|u(\xi)|^{p}\d\xi
     -\int_{\Lambda}|u(\xi)|^{p-2}u(\xi)\cdot P_{\frac{t}{n}}u(\xi)\d\xi\right)\d t
     +\phi(t)\|u\|_{n}^{2}\\
   & \leq -C_{2}\|u\|_{n}^{2}.
\end{align*}
Then we have
\begin{equation*}
  2_{V^{*}}\langle A(t,u),T_{n}u\rangle_{V}+\|B(t)\|^{2}_{L_{2}(U,H_{n})}\leq
  -2C_{2}\|u\|_{n}^{2}+\hat{M}.
\end{equation*}
That is, (H5) holds.

The proof is complete.
\end{proof}

\section*{Acknowledgements}
This work is partially supported by NSFC Grants 11522104, 11871132, and the startup and
Xinghai Jieqing funds from Dalian University of Technology.

\end{document}